\documentclass[11pt,a4paper,oneside]{amsart}
\newcounter{commentcounter}

\usepackage{amsmath} % Lots of maths functionality
\usepackage{amssymb} % Maths symbols
\usepackage{amsthm} % Maths environments: \begin{proof}, etc.
\usepackage{stmaryrd} % [[ brackets
\usepackage[english]{babel} % Language and hyphenation.
\usepackage[font=small,justification=centering]{caption} % More flexibility for captioning figures
\usepackage[nodayofweek]{datetime}
\usepackage[shortlabels]{enumitem} % Change enumeration labelling with \begin{enumerate}[a)] etc.
\usepackage[T1]{fontenc} % For font encoding, to allow accents, copy & paste, inequality signs, etc. to all work nicely.
\usepackage[utf8]{inputenc} % To be loaded after fontenc, also for encoding.
\usepackage{ifthen} % For Dani's \begin{com} environment and \numberedtheorem
\usepackage{mathabx} % Contains \pnest symbol and more. Clashes with accents for \ring
\usepackage{mathtools} % Uses amsmath, fixes quirks and adds functionality.
\usepackage[dvipsnames]{xcolor} % Allow links to have colour. Needs to be before hyperref and tikz-cd
\usepackage[pdftex,  colorlinks=true, pagebackref=true]{hyperref} % Makes references and citations into links
    \hypersetup{urlcolor=RoyalBlue, linkcolor=RoyalBlue,  citecolor=black}
\usepackage{setspace} % Allows \onehalfspacing etc. for changing gaps between lines
\usepackage{tikz-cd} % Commutative diagrams
\usepackage{xfrac} % Nicer quotients 
\usepackage[capitalize]{cleveref} % use \Cref{} for instead of X~\ref{} 

\renewcommand*{\backref}[1]{}
\renewcommand*{\backrefalt}[4]
{
    \ifcase #1
        No citation in the text.
    \or
        Cited on Page #2.
    \else
        Cited on Pages #2.
    \fi
}

\makeatletter
\@namedef{subjclassname@1991}{Mathematical subject classification 1991}
\@namedef{subjclassname@2000}{Mathematical subject classification 2000}
\@namedef{subjclassname@2010}{Mathematical subject classification 2010}
\@namedef{subjclassname@2020}{Mathematical subject classification 2020}
\makeatother

\newtheorem{thm}{Theorem}[section]
\crefname{thm}{Theorem}{Theorems}
\newtheorem{lemma}[thm]{Lemma}
\crefname{lemma}{Lemma}{Lemmata}
\newtheorem{corollary}[thm]{Corollary}
\newtheorem{prop}[thm]{Proposition}

\newtheorem{claim}[thm]{Claim}

\newtheorem{question}[thm]{Question}

\numberwithin{equation}{thm}

\newtheorem{thmx}{Theorem}
\crefname{thmx}{Theorem}{Theorems}
\newtheorem{corx}[thmx]{Corollary}

\theoremstyle{definition}
\newtheorem{defn}[thm]{Definition}
\newtheorem{remark}[thm]{Remark}

\newtheorem{example}[thm]{Example}

\theoremstyle{plain}
\newtheorem*{ConjSinger}{The Singer Conjecture}
\newtheorem*{ConjAtiyah}{The Atiyah Conjecture}

    \newtheoremstyle{TheoremNum}
        {8.0pt plus 2.0pt minus 4.0pt}{8.0pt plus 2.0pt minus 4.0pt} %%% space between body and thm
        {\itshape} %%% Thm body font
        {-7pt} %%% Indent amount (empty = no indent)
        {\bfseries} %%% Thm head font
        {.} %%% Punctuation after thm head
        { }  %%% Space after thm head
        {\thmname{#1}\thmnote{ \bfseries #3}}%%% Thm head spec
    \theoremstyle{TheoremNum}
    \newtheorem{duplicate}{}

\newcommand*{\claimproofname}{My proof}
\newenvironment{claimproof}[1][\claimproofname]{\begin{proof}[#1]}{\end{proof}}

\newcommand*{\markproofname}{Proof}
\newenvironment{markproof}[1][\markproofname]{\begin{proof}[#1]}{\end{proof}}

\DeclareMathOperator{\Aut}{\mathrm{Aut}}

\DeclareMathOperator{\Hom}{\mathrm{Hom}}

\DeclareMathOperator{\Tor}{\mathrm{Tor}}
\DeclareMathOperator{\tr}{\mathrm{tr}}
\DeclareMathOperator{\im}{\mathrm{im}}

\DeclareMathOperator{\dom}{\mathrm{dom}}

\newcommand{\cald}{{\mathcal{D}}}

\newcommand{\caln}{{\mathcal{N}}}

\newcommand{\calr}{{\mathcal{R}}}

\newcommand{\calu}{{\mathcal{U}}}

\newcommand{\calx}{{\mathcal{X}}}

\newcommand{\SU}{\mathrm{SU}}

\newcommand{\SO}{\mathrm{SO}}

\DeclareMathOperator{\id}{id}
\newcommand{\onto}{\twoheadrightarrow}
\def\iff{if and only if }

\newcommand{\cd}{\mathrm{cd}}

\newcommand{\nov}[3]{{\widehat{#1 #2}^{#3}}}
\def\Z{\mathbb{Z}}

\newcommand{\NN}{\mathbb{N}}
\newcommand{\ZZ}{\mathbb{Z}}
\newcommand{\CC}{\mathbb{C}}
\newcommand{\RR}{\mathbb{R}}
\newcommand{\QQ}{\mathbb{Q}}
\newcommand{\FF}{\mathbb{F}}

\usepackage{tikz}
\usetikzlibrary{arrows,quotes}
\tikzstyle{blackNode}=[fill=black, draw=black, shape=circle]

\DeclareMathOperator{\nul}{null}
\DeclareMathOperator{\init}{init}
\DeclareMathOperator{\wrat}{WRat}
\DeclareMathOperator{\prewrat}{WRat_0}

\newcommand{\apa}{\approx}

\newcounter{dawidcomments}

\title{BNSR invariants and \texorpdfstring{$\ell^2$}{l2}-homology}

\usepackage[foot]{amsaddr}
\author{Sam Hughes}
\author{Dawid Kielak}
\address{Mathematical Institute, Andrew Wiles Building, Observatory Quarter, University of Oxford, Oxford OX2 6GG, UK}
\email{sam.hughes@maths.ox.ac.uk}
\email{kielak@maths.ox.ac.uk}

\date{\today}

\begin{document}

\begin{abstract}
	We prove that if the $n$th $\ell^2$-Betti number of a group is non-zero then its $n$th BNSR invariant over $\mathbb{Q}$ is empty, under suitable finiteness conditions.  We apply this to answer questions of Friedl--Vidussi and Llosa Isenrich--Py about aspherical K\"ahler manifolds, to verify some cases of the Singer Conjecture, and to compute certain BNSR invariants of poly-free and poly-surface groups.
\end{abstract}	
	
\dedicatory{Dedicated to the memory of Peter A.~Linnell.}	
	
\maketitle

\section{Introduction}

It was shown by L\"uck \cite{Lueck1994} that non-vanishing $\ell^2$-Betti numbers $b_n^{(2)}$ obstruct finite CW complexes being homotopic to mapping tori.  In the case of the classifying space of a group $G$, this can be restated in terms of finiteness properties of the kernel of a homomorphism $G\onto \Z$ (see \cite[Theorem~6.63]{Lueck2002} for the statement regarding connected CW complexes).

\begin{thm}[L\"uck]Let $G$ be a group of type $\mathsf{F}_n$, that is admitting a classifying space with finite $n$-skeleton, and let $\varphi\colon G\onto \Z$ be an epimorphism.  If $b_n^{(2)}(G)\neq 0$, then $\ker\varphi$ is not of type $\mathsf{F}_n$.\end{thm}

Bieri, Neumann, Renz, and Strebel \cite{BieriNeumannStrebel1987,BieriRenz1988} introduced a family of group invariants which are intimately related to the finiteness properties of kernels of homomorphisms $G\to\RR$.  In spite of this, the precise relation between the $n$th $\ell^2$-Betti number and the $n$th BNSR invariants $\Sigma^n$ has remained mysterious.  Indeed, the question is rather subtle:
for $n=1$, Brown \cite{Brown1987} gave a characterisation of the integral characters $G \onto \Z$ in $\Sigma^1(G)$ as those that are induced by $G$ splitting as an ascending HNN extension over a finitely generated group; such a splitting gives $G$ the structure of an algebraic mapping torus. There is however no such characterisation for $n>1$; the best result in this direction seems to be a theorem of Renz for $n=2$  \cite{Renz1989}.
Our main theorems clarify the link between the BNSR invariants and $\ell^2$-Betti numbers in both the cases of groups and spaces.

\begin{duplicate}[\Cref{thm groups}]
Let $G$ be a group of type $\mathsf{FP}_n(\QQ)$.  If $b_n^{(2)}(G)\neq0$, then $\Sigma^n(G;\QQ)=\emptyset$.
\end{duplicate}

Recall that type $\mathsf{FP}_n$ is a homological version of type $\mathsf F_n$; we give a formal definition in \cref{sec BNSR}.

The theorem was already known for RFRS groups (see \cite[Theorem~5.10]{HughesKielak2022} building on \cite{Kielak2020RFRS} and \cite{Fisher2021}). It is also significantly easier to establish this theorem when $G$ is torsion free and known to satisfy the Atiyah conjecture.  We have an analogous statement for CW complexes.

\setcounter{thmx}{1}
\begin{thmx}\label{thm spaces}
Let $X$ be a connected CW complex with finite $n$-skeleton.  If $b_n^{(2)}(\widetilde X;\pi_1X)\neq0$, then $\Sigma^n(X;\QQ)=\emptyset$.
\end{thmx}

A special case of the above result with $X$ being a closed smooth manifold can be deduced from work of Farber \cite[Corollary~2]{Farber2000} (based on \cite{NovikovShubin1986}).  The methods used here are of a very different nature: they are ring-theoretic rather than analytic, as is the case of Farber.

For a commutative ring $R$, recall that an \emph{$n$-dimensional Poincar\'e duality complex $M$ over $R$} (or a \emph{$\mathsf{PD}_n(R)$-complex}) is a finitely dominated connected CW complex with a distinguished class $[M]$ in the $n$th homology group $H_n(M;D)$, where $D$ is an $RG$ module isomorphic to $R$ as a group, such that the cap product
\[[M]\frown-\colon H^k(M;A)\to H_{n-k}(M;A \otimes_{R} D)\]
 is an isomorphism for all choices of local coefficients $A$ over $R$, where the action on $A \otimes_R D$ is diagonal.  Similarly, a \emph{$\mathsf{PD}_n(R)$-group} is a group of type $\mathsf{FP}(R)$ with cohomological dimension $\cd_R(G)=n$ for which $D=H^n(G;RG)$ is isomorphic to $R$ as an abelian group and $H^i(G;M)\cong H_{n-i}(G;M\otimes_R D)$ for all $RG$-modules $M$ (here $M\otimes_R D$ is equipped with the diagonal action).

Note that the fundamental group of an aspherical $\mathsf{PD}_n(R)$-complex is a finitely presented $\mathsf{PD}_n(R)$-group, but for $n\geq4$ there exist examples of $\mathsf{PD}_n(R)$-groups which are not finitely presented \cite{Davis1998,Davis1998c} (in fact there are uncountably many such groups \cite{Leary2018,KrophollerLearySoroko2020}).

The motivation for the authors to begin investigating the connection between vanishing of the $n$th BNSR invariants (over $\QQ$) and non-vanishing of the $n$th $\ell^2$-Betti number was provided by the following question which appeared in \cite[remark after Proposition~3.4]{FriedlVidussi2021} and \cite[Question~7]{IsenrichPy2022}.

\begin{question}[Friedl--Vidussi, Llosa Isenrich--Py]\label{q motivation}
	Let $M$ be a closed aspherical K{\"a}hler $2n$-manifold.  If $\chi(M)\neq0$, is $\Sigma^n(\pi_1M)$ empty?
\end{question}

In fact we are able to give an affirmative answer to the more general question obtained  by dropping the hypotheses `aspherical' and `K{\"a}hler' and instead studying the BNSR invariants of the manifold.

\begin{duplicate}[\Cref{cor manifolds}]
	Let $M$ be a closed connected $2n$-manifold or (more generally) a finite $\mathsf{PD}_{2n}(\QQ)$-complex.
	If $\chi(M)\neq 0$,
	then $\Sigma^n(M)=\Sigma^n(M;\QQ)=\emptyset$.
	In particular, if $M$ is additionally aspherical, then
	$\Sigma^n(\pi_1M)=\Sigma^n(\pi_1M;\QQ)=\emptyset$.
\end{duplicate}

\begin{remark}
	A near-identical proof yields the following result: If $G$ is a $\mathsf{PD}_{2n}(\QQ)$-group such that $\chi(G)\neq0$, then $\Sigma^n(G;\QQ)=\emptyset$.
\end{remark}

\smallskip
The \emph{Singer Conjecture} is one of the major unresolved problems regarding $\ell^2$-Betti numbers and the topology of manifolds.  For more information the reader is directed to \cite[Chapter~11]{Lueck2002}.

\begin{ConjSinger}
If $M$ is a closed aspherical $n$-manifold, then $b^{(2)}_p(\widetilde{M};\pi_1 M)=0$ for all $p\neq n/2$.
\end{ConjSinger}

The Singer Conjecture can clearly be generalised to $\mathsf{PD}_n(\QQ)$-groups; in this setting we have the following corollary.

\begin{duplicate}[\Cref{cor Singer}]
Let $G$ be a $\mathsf{PD}_n(\QQ)$-group and let $k=\lceil n/2\rceil-1$.  If $\Sigma^{k}(G;\QQ)\neq \emptyset$, then the Singer Conjecture holds for $G$.
\end{duplicate}

\subsection*{The idea behind the proofs}

The main technical result is \cref{prop complexes}, in which we prove that if a chain complex over a group ring $\QQ G$ has vanishing Novikov homology up to some dimension, then its $\ell^2$-homology vanishes as well, up to the same dimension. Novikov homology has been shown by Sikorav to encode the BNSR invariants, and $\ell^2$-homology can also be understood algebraically, as homology with coefficients in the algebra $\calu G$ of operators affiliated to the group von Neumann algebra of $G$. Hence, we are trying to show that if we have a partial chain contraction defined over the Novikov ring, then we also have one over $\calu G$.
We do this in two steps: first, in \cref{chain contractions} we show that one does not need to work with the whole Novikov ring, but instead it is enough to use the division closure of $\QQ G$ inside the Novikov ring. Then, in \cref{Lambdas}, we show that there is a ring that we call the ring of \emph{weakly rational elements} of $\calu G$, that is simultaneously a subring of $\calu G$, and an overring of the division closure of $\QQ G$ in the Novikov ring.

The construction of the ring of weakly rational elements is where the blood, sweat, tears, and toil went.
In an ideal world, for example when $G$ is torsion free and satisfies the Atiyah conjecture, one argues as follows: the Novikov ring corresponding to $\phi \colon G \to \Z$ consists naturally of twisted Laurent power series with coefficients in $\QQ \ker \phi$, and so in particular in the Linnell skew field $\cald \ker \phi$, a particularly useful subring of $\calu \ker \phi$. Since twisted Laurent polynomials over $\cald \ker \phi$ satisfy the Ore condition, one can look at the rational functions, and it is not hard to show that, on the one hand, such rational functions contain the division closure of $\QQ G$ inside the Novikov ring, and on the other, are naturally contained in $\calu G$, even though the ring of Laurent power series is not. Unfortunately, the authors were unable to show that the twisted Laurent polynomial ring satisfies the Ore condition -- this remains an open problem. Instead, we show that this ring satisfies an approximate version of the Ore condition, which allows us to form a ring of `approximate rational functions', that is precisely the ring of weakly rational elements.

\subsection*{Outline of the paper}
In \Cref{sec prelims}, we give the relevant background on $\ell^2$-homology, BNSR theory, some ring theoretic tools, and several rings and algebras related to a countable discrete group that we will need in our proofs, with particular emphasis on group von Neumann algebras.

In \Cref{sec passing vanishing}, we prove that vanishing of Novikov homology can be passed to a division-closed subring.
We highlight two results of independent interest: The first is \Cref{nov DivClos is RatClos} that states that for a discrete group $G$ the division closure of $RG$ in a Novikov ring $\nov{R}{G}{\varphi}$ is equal to its rational closure.  The second is \Cref{inductive chain contraction} that gives a method for taking chain contractions over $\nov{R}{G}{\varphi}$ and rebuilding them over the division closure of $RG$.

\Cref{sec ring hom}
is the technical heart of the paper. Here we introduce the ring of weakly rational elements of $\calu G$, and prove that it contains the division closure of $\QQ G$ inside of the Novikov ring.

In \Cref{sec proofs}, we prove each of the results \ref{thm groups} - \ref{cor Singer}.

In \Cref{sec Atiyah}, we prove a characteristic $p$ version of \Cref{thm groups} and \Cref{thm spaces} (see \Cref{thm char p}) for groups whose group rings admit embeddings into certain universal skew fields.  We also give a simplified proof of \ref{thm groups} - \ref{cor Singer} when $G$ is torsion free and satisfies the Atiyah Conjecture.

In \Cref{sec ex}, we detail a number of example applications of our results.  In \Cref{polycule} we prove that $\Sigma^n(G;\QQ)=\emptyset$ when $G$ is a poly-elementarily-free group of length $n$.  This class includes poly-free and poly-surface groups.  The result generalises \cite[Proposition~1.5]{KochloukovaVidussi2022} where the authors compute $\Sigma^2(G)$ for groups isomorphic to $F_n\rtimes F_m$ or $\pi_1\Sigma_g\rtimes\pi_1\Sigma_h$ with $n,m,g,h>1$ (see the first part of \cite[Theorem~6.1]{KrophollerWalsh2019} for the case $F_2\rtimes F_m$); here $\Sigma_g$ is the closed orientable connected surface of genus $g$. % Our main contribution  involves computing the $\ell^2$-Betti numbers of a poly-elementarily-free group (\Cref{Bno poly elefree}).
  We also highlight how our \Cref{thm groups} applies to real and complex hyperbolic lattices.

\subsection*{Acknowledgements}
This work has received funding from the European Research Council (ERC) under the European Union's Horizon 2020 research and innovation programme (Grant agreement No. 850930).

  The authors thank Claudio Llosa Isenrich for alerting them to  \Cref{q motivation}.

\section{Preliminaries}\label{sec prelims}

Throughout, all rings are assumed to be associative and unital. All groups will be discrete, and modules will be right modules, unless stated otherwise.

\subsection{Group von Neumann algebras}\label{sec NG}
Let $G$ be a countable group and let $\ell^2 G$ denote the Hilbert space of square summable formal sums of elements of $G$ with complex coefficients, that is the space of expressions
\[\sum_{g\in G}\lambda_g g \, \text{ such that }\sum_{g\in G}|\lambda_g|^2<\infty\text{ where }\lambda_g\in \CC. \]
The group $G$ acts on $\ell^2 G$ by right multiplication.
\begin{defn}[Group von Neumann algebra]
We define the \emph{group von Neumann algebra} $\caln G$ of $G$ to be the algebra of $G$-equivariant bounded operators $\ell^2 G\to\ell^2 G$.
\end{defn}

Since $G$ acts on $\ell^2 G$ from the right, it is natural to view $\caln G$ as operating from the left on $\ell^2 G$. In particular, the copy of $\CC G$ in $\caln G$ acts on $\ell^2 G$ from the left. Since $\caln G$ contains a copy of $\CC G$, multiplication turns it into a $\CC G$-bimodule.

One of the key features of group von Neumann algebras is that for every closed $G$-invariant subspace $V$ of $\ell^2 G$, there is an associated projection $\pi_V \in \caln G$; it is a self-adjoint idempotent with $\im \pi_V = V$ and $\ker \pi_V = V^\perp$.

Another important class of operators in $\caln G$ are \emph{partial isometries}, that is operators $u$ such that $u^*u = \pi_{(\ker u )^\perp}$. It is easy to see that this implies that $\im u$ is closed and that $uu^* = \pi_{\im u}$.

%\begin{lemma}
%	\label{partial isom existence}
%	For every $x \in \caln G$, there exists a partial isometry $u \in \caln G$ such that $u^*u = \pi_{{\ker x}^\perp}$ and $uu^* = \pi_{\overline {\im x}}$.
%\end{lemma}

\begin{defn}[Polar decomposition]
	The \emph{polar decomposition} of  $s\in \caln G$ is a canonical factorization $s=vp$ where $v$ is a partial isometry and $p$ is a non-negative operator, and both lie in $\caln G$. Furthermore, $v^*v$ is the projection onto $\ker s^\perp$, and $vv^*$ is the projection onto the closure of $\im s$, whence it easily follows that $\ker v = \ker s$ and $\im v$ is the closure of $\im s$.
\end{defn}

The group von Neumann algebra comes equipped with the \emph{von Neumann trace} $\tr$. There is a canonical dimension function $\dim_{\caln G}$, taking values in $[0,\infty]$, defined on $\caln G$-modules called the \emph{von Neumann dimension} - the precise definition and basic properties we use can be found in \cite[Chapter~6]{Lueck2002}. One can also define such a dimension for closed $G$-invariant subspaces of $\ell^2 G$, in which case it is equal to the von Neumann trace of the corresponding projection. Since the only projection with zero trace is $0$, the only  $G$-invariant subspace of $\ell^2 G$ of von Neumann dimension zero is the trivial subspace.

We will also use the following theorem.

\begin{thm}[Linnell]\emph{\cite[Theorem~4]{Linnell1992}}\label{Linnell zero divs}
Let $H\trianglelefteq G$ be such that $G/H$ is right orderable with total order $\leq$ and let $\varphi\colon G\twoheadrightarrow G/H$ be the natural epimorphism.  Let $T$ be a transveral for $H$ in $G$, let $x\in\ell^2 G$, and write $x=\sum_{t\in T}x_t t$ where $x_t\in\ell^2 H$ for all $t\in T$.  Suppose that there exists $t_0\in T$ such that $x_{t}=0$ for $t$ with $\varphi(t)<\varphi(t_0)$.  If $x_{t_0} y'\neq 0$ for all non-zero $y'\in\caln H$, then $xy\neq 0$ for all non-zero $y\in\ell^2 G$.
\end{thm}

\subsection{\texorpdfstring{$\ell^2$}{l\texttwosuperior}-homology and Betti numbers}\label{sec l2Bno}

\begin{defn}[$\ell^2$-homology and Betti numbers]
Let $X$ be a $G$-CW complex. We define the \emph{$\ell^2$-homology} of $X$ with respect to $G$ as
\[H_p^G(X;\caln G)\coloneqq H_p\big(C_\bullet(X;\QQ)\otimes_{\QQ G} \caln G \big), \]
where $C_\bullet(X;\QQ)$ is the cellular chain complex of $X$ with rational coefficients considered as a complex of free $\QQ G$-modules.

We define the \emph{$\ell^2$-Betti numbers} of $X$ with respect to $G$ to be
\[b_p^{(2)}(X;G)\coloneqq \dim_{\caln G}H_p^G(X;\caln G). \]
The $\ell^2$-Betti numbers of a group $G$ are defined to be the $\ell^2$-Betti numbers of the universal free $G$-space $EG$.  A number of properties of $\ell^2$-Betti numbers can be found in \cite[Theorem~6.54]{Lueck2002}.
\end{defn}

The $\caln G$-dimension does not change when one qoutients an $\caln G$-module by its torsion.

There are two other ways of defining $\ell^2$-homology; for details, see \cite{Lueck1998}. One can look at the reduced homology of the complex
\[C_\bullet(X;\QQ)\otimes_{\QQ G} \ell^2 G,\]
 where \emph{reduced} means that we divide kernels by closures of images. This way the homology groups are actually closed subspaces of powers of $\ell^2 G$, and one can look at the von Neumann dimension of such a subspace. The dimension coincides with the $\ell^2$-Betti number, as defined above.

The third way of defining $\ell^2$-homology involves the algebra of affiliated operators $\calu G$, which we will define in a moment. One can look at the homology of $C_\bullet(X;\QQ)\otimes_{\QQ G} \calu G$, which ends up being equal to the homology of $C_\bullet(X;\QQ)\otimes_{\QQ G} \caln G$ tensored with $\calu G$. This has essentially the effect of quotienting the $\ell^2$-homology as we defined it by the torsion. There is again a notion of a dimension for $\calu G$-modules, and again one obtains the same Betti numbers.

\subsection{Ore localisation}\label{sec Ore}
In this section we will describe an analogue of localisation for non-commutative rings.

\begin{defn}
	Let $R$ be a ring. An element $x \in R$ is a \emph{zero-divisor} if $x \neq 0$, and $xy = 0$ or $yx = 0$ for some non-zero $y \in R$. A non-zero element that is not a zero-divisor will be called \emph{regular}.
\end{defn}

\begin{defn}[Right Ore condition]
Let $R$ be a ring and $S\subseteq R$ a multiplicatively closed subset consisting of regular elements.  The pair $(R,S)$ satisfies the \emph{right Ore condition} if
for every $r\in R$ and $s\in S$ there are elements $r'\in R$ and $s'\in S$ satisfying $rs'= sr'$.
\end{defn}

\begin{defn}[Right Ore localisation]
If $(R,S)$ satisfies the right Ore condition we may define the \emph{right Ore localisation}, denoted $RS^{-1}$, to be the following ring.
Elements are represented by pairs $(r,s)\in R\times S$ up to the following equivalence relation: $(r,s)\sim(r',s')$ if and only if there exists $u,u'\in R$ such that the equations $ru=r'u'$ and $su=s'u'$ hold, and $su=s'u'$ belongs to $S$.
The addition is given by
\[(r,s)+(r',s')=(rc+r'd,t), \text{ where }t=sc=s'd\in S,  \]
and the multiplication is given by
\[(r,s)(r',s')=(rc,s't), \text{ where } sc=r't \text{ with }t\in S.\]
\end{defn}

There is a natural ring homomorphism $R\to RS^{-1}$ defined by $r\mapsto(r,1)$.  For more information on this construction the reader is referred to \cite[Section~4.4]{Passman1985}. Note that there is also an analogously defined \emph{left} Ore condition.

\subsection{The algebra of affiliated operators}\label{sec UG}

\begin{defn}[Affiliated operators]
We say that an operator
\[f\colon\dom(f) \to\ell^2 G\]
with $\dom(f)\subseteq\ell^2 G$ is \emph{affiliated} (to $\caln G$) if  $f$ is densely defined with domain $\dom(f)$, is closed, and is a $G$-operator, that is, $\dom(f)$ is a linear $G$-invariant subspace and $f(x).g=f(x.g)$ for all $g\in G$ (recall that $G$ acts on $\ell^2 G$ on the right).

The set of all operators affiliated to $\caln G$ forms the \emph{algebra of affiliated operators} $\calu G$ of $G$.
\end{defn}

Since an adjoint of a densely defined closed operator is densely defined and closed, every $x \in \calu G$ has a well-defined adjoint $x^* \in \calu G$.

Note that we have inclusions of $\QQ G$-modules \[\QQ G\rightarrowtail\CC G\rightarrowtail \caln G\rightarrowtail \calu G.\]
%
%As with $\caln G$, the algebra $\calu G$ has an associated \emph{von Neumann dimension}, denoted by $\dim_{\calu G}$, which takes values in $[0,\infty]$.  The actual definition is technical and need not concern us -- we refer the interested reader to \cite[Definition~8.28]{Lueck2002}.
%
For more information on these constructions the reader is referred to \cite{Lueck2002} -- specifically Theorem~8.22 and more generally Chapter~8.  We highlight one theorem of special importance to us.

\begin{thm}\emph{\cite[Theorem~8.22(1)]{Lueck2002}}\label{Ore for von Neumann}
The set $S$ of regular elements of $\caln G$ forms a right Ore set.  Moreover, $\calu G$ is canonically isomorphic to $(\caln G) S^{-1}$.
\end{thm}

The algebra $\calu G$ is von Neumann regular, with explicit control on what the partial inverses look like. In particular, for every $x \in \calu G$ there exists a canonical $x^\dagger \in \calu G$ such that $x x^\dagger = \pi_{\overline{\im x}}$ as an affiliated operator. Note that in reality the composition $x \circ x^\dagger$ is defined only on $\im x \oplus (\im x)^\perp$, and coincides with $\pi_{\overline{\im x}}$ on this subspace. However, the graph of this composition is not closed, and hence the composition is not an affiliated operator. It can be extended to one, and this extension is precisely $\pi_{\overline{\im x}}$. Similarly, we have $x^\dagger x= \pi_{(\ker x)^\perp}$. We also have $\ker x^\dagger = (\im x)^\perp$ and $\overline{\im x^\dagger} = (\ker x)^\perp$.

The partial inverse can be constructed directly, as in the proof of \cite[Lemma 8.3(2)]{Lueck1994}, or its existence can be deduced algebraically, as in \cite[Proposition 3.2]{Jaikin2019}.

%It follows that $\calu G$ is flat as an $\caln G$-module.  In particular (see \cite[Theorem~8.29)]{Lueck2002} for details), for any $\caln G$-module $M$ we have \[\dim_{\calu G}M\otimes_{\caln G}\calu G=\dim_{\caln G}M.\]
%Moreover, $\dim_{\calu G}$ is \emph{faithful}, that is, a $\calu G$-module $M$ is $0$ if and only if $\dim_{\calu G}M=0$.

\subsection{Division and rational closures}

\begin{defn}[Division and rational closure]
	Let $R$ be a ring and $S$ a subring.
	We say that $S$ is \emph{division closed} if every element of $S$ invertible over $R$ is invertible over $S$. We say that $S$ is \emph{rationally closed} if every finite square matrix over  $S$ invertible over $R$ is invertible over $S$.
	
	Define the \emph{division closure of $S$ in $R$}, denoted by $\cald(S\subset R)$,  to be the smallest division-closed subring  of $R$ containing $S$. Define the \emph{rational closure of $S$ in $R$}, denoted by $\calr(S\subset R)$, to be the smallest rationally closed subring  of $R$ containing $S$.
\end{defn}

\subsection{Twisted polynomial rings}\label{sec twist}

\begin{defn}[Twisted polynomial ring]
Let $R$ be a ring %with group of units $R^\times$
 and let $R[t^{\pm1}]$ be the abelian group of Laurent polynomials over $R$.  Given a homomorphism $\nu\colon\Z\to \Aut(R)$ we may endow $R[t^{\pm1}]$ with a non-commutative multiplication given by
\[xt^m\cdot yt^n=x \nu(t^m)(y)t^{m+n}, \]
and extended linearly. We call this new ring the \emph{ring of twisted Laurent polynomials over $R$ with respect to $\nu$} or simply a \emph{twisted Laurent polynomial ring}.
\end{defn}

\begin{remark}\label{rem twisted RHZ}
Suppose that $G$ splits as $H\rtimes\Z$. The group ring $RG$ is isomorphic to a ring of twisted Laurent polynomials $RH[t^{\pm 1}]$ in a natural way.
\end{remark}

\subsection{BNSR invariants}\label{sec BNSR}
In this section we follow the treatment of Farber--Geoghegan--Sch\"utz \cite[Section~2]{FarberGeogheganSchuetz2010}.    Note that the authors only give statements and proofs for BNSR invariants over the ring $\ZZ$, however the arguments easily generalise to arbitrary rings $R$.

Let $R$ be a ring.  A monoid $M$ is of \emph{type $\mathsf{FP}_n(R)$} if there exists a projective resolution $P_\bullet\to R$ of the trivial $RM$-module $R$ with $P_i$ finitely generated for $i\leq n$.  Let $G$ be a finitely generated group.  Define $S(G)=\Hom(G;\RR)\smallsetminus\{0\}$.
% divided by positive homothety.
%We will often identify $S(G)$ with the unit sphere in $\Hom(G;\RR)$.
Given $\varphi\in S(G)$, define a submonoid of $G$ by
\[{G}_{\varphi}\coloneqq\left\{ g\in G : \varphi(g)\geq 0\right\}.\]

\begin{defn}[Homological BNSR invariants of groups]
Let $R$ be a ring and $G$ be a group of type $\mathsf{FP_n}(R)$ for $n\in\NN\cup\{\infty\}$. For $\varphi\in S(G)$ we write $\varphi\in \Sigma^n(G;R)$ if $G_\varphi$ is of type $\mathsf{FP}_n(R)$.
\end{defn}

For $R$ a ring, $n\in\NN$, and $C_\bullet$ a chain complex over $R$, we say that $C_\bullet$ has \emph{finite $n$-type} if there exists a finitely generated chain complex $P_\bullet$ of projective $R$-modules and a chain map $f\colon C_\bullet\to P_\bullet$ such that $f_i\colon H_i(C_\bullet)\to H_i(P_\bullet)$ is an isomorphism for $i<n$ and an epimorphism for $i=n$.  In this case we call $f$ an \emph{$n$-equivalence}.

Let $C_\bullet$ be a chain complex of $RG$-modules and let $k\in\NN$.  We define
\[\Sigma^k(C_\bullet;R)\coloneqq\{\varphi\in S(G) : C_\bullet \text{ is of finite $k$-type over } RG_\varphi\}.\]

\begin{defn}[Homological BNSR invariants of spaces]
Let $X$ be a connected CW complex with finite $n$-skeleton.  For $k\leq n$ the \emph{$k$-th (homological) BNSR invariant of $X$ over $R$} is defined to be
\[\Sigma^k(X;R)\coloneqq \Sigma^k(C_\bullet(\widetilde{X};R);R), \]
where $\widetilde{X}$ is the universal cover of $X$, and $C_\bullet(-;R)$ denotes the cellular chain complex with coefficients in $R$.
\end{defn}

\begin{defn}[Homotopic BNSR invariants of groups and spaces via closed $1$-forms]
Let $X$ be a connected CW complex with finite $n$-skeleton and with $\pi_1 X=G$.  Let $\omega$ be a closed $1$-form on $X$ representing $\varphi\in S(G)$ and let $0\leq k \leq n$.  We define the \emph{$k$-th homotopic BNSR invariant of $X$} to be the set $\Sigma^k(X)\subseteq S(G)$ defined as follows:  $\varphi\in\Sigma^k(X)$ if there exists an $\epsilon>0$ and a cellular homotopy $H\colon X^{(k)}\times I\to X$ such that $H(x,0)=x$ and
\[\int_{\gamma_x}\omega \geq \epsilon\]
 for all $x\in X^{(k)}$, where $\gamma_x\colon[0,1]\to X$ is given by $\gamma_x(t)=H(x,t)$.
For a group $G$ we take $\Sigma^n(G)\coloneqq\Sigma^n(BG)$.
\end{defn}

The following theorem essentially combines Theorem~4, Proposition~3, and Corollaries~1 and 2 of \cite{FarberGeogheganSchuetz2010}; we have taken the liberty to state these results over a general ring $R$.

\begin{thm}[Basic properties of BNSR invariants] \label{thm props BNSR}
Let $X$ be a connected CW complex with finite $n$-skeleton and with $\pi_1 X=G$, and let $R$ be a ring.  Then,
\begin{enumerate}
    \item $\Sigma^k(X)$ and $\Sigma^k(X;R)$ are open subsets of $S(G)$; \label{thm props BNSR open}
    \item $\Sigma^k(X)\subseteq \Sigma^k(X;R)$; \label{thm props BNSR htpy subset hmlgy}
    \item if $\widetilde{X}$ is $k$-connected then
    \[\Sigma^k(X)=\Sigma^k(G)\text{ and }\Sigma^k(X;R)=\Sigma^k(G;R)\]
    and
    \[\Sigma^{k+1}(X)\subseteq\Sigma^{k+1}(G)\text{ and }\Sigma^{k+1}(X;R)\subseteq\Sigma^{k+1}(G;R);\]
    \item if $X$ is finite, then for all $k\geq\dim X$ we have \[\Sigma^k(X)=\Sigma^{\dim X}(X)\text{ and }\Sigma^k(X;R)=\Sigma^{\dim X}(X;R).\]
    %\item if $Y$ is a finite connected cover of $X$ and $\varphi \in \Sigma^k(X;R)$ then $\varphi\vert_{\pi_1(Y)} \in \Sigma^k(Y;R)$.
\end{enumerate}
\end{thm}

\subsection{Novikov--Sikorav homology}\label{sec Nov}

\begin{defn}[Novikov--Sikorav ring]
Let $G$ be a finitely generated group, let $\varphi\in S(G)$, and define the \emph{Novikov--Sikorav ring} of $G$ with respect to $\varphi$ to be
\[\nov{R}{G}{\varphi}\coloneqq\left\{\sum_{g\in G}n_g g \, : \,  \big\vert\{g : n_g\neq 0\text{ and }\varphi(g)<t\}\big\vert <\infty\text{ for all }t\in\RR\right\}.\]
\end{defn}

\begin{defn}[Truncation]
	Given an element $x=\sum_{g\in G}n_g g \in \nov{R}{G}{\varphi}$ and a real number $r$, we define the \emph{truncation} of $x$ at $r$ to be the element $\sum_{g\in G}n'_g g$ where
	\[
	n'_g = \left\{ \begin{array}{ccl}
	n_g & \textrm{ if } & \varphi(g) \leqslant r \\
	0 & \textrm{ if } & \varphi(g) > r.
	\end{array} \right.
	\]
	Note that the truncation lies in $RG$.
\end{defn}

\begin{defn}[Novikov--Sikorav homology]
Let $X$ be a CW complex with $\pi_1 X=G$ and let $\varphi\in S(G)$.
The \emph{Novikov--Sikorav homology}
\[H_\bullet^G(\widetilde{X};\nov{R}{G}{\varphi})\]
 of $X$ is the $G$-equivariant homology of $\widetilde{X}$, the universal cover of $X$, with non-trivial coefficients $\nov{R}{G}{\varphi}$, that is, the homology of the chain complex  $C_\bullet(\widetilde{X};R) \otimes_{RG} \nov{R}{G}{\varphi}$.
\end{defn}

\begin{remark}
	\label{BNS finite index}
Straight from the definition, it is easy to see that if $H$ is a finite index subgroup of $G$ and $H_i^G(\widetilde X;\nov{R}{G}{\varphi})=0$, then
$H_i^H(\widetilde X;\nov{R}{H}{\varphi\vert_H})=0$ as well, since $\nov{R}{G}{\varphi} = R G \otimes_{R H} \nov{R}{H}{\varphi\vert_H}$.
\end{remark}

The following result has an essentially identical proof as \cite[Theorem~5.3]{Fisher2021} (see also \cite[Proposition~5]{FarberGeogheganSchuetz2010}, \cite[Appendix]{Bieri2007}, \cite[Theorem~3.11]{Kielak2020RFRS}, and \cite{Sikorav1987}).

\begin{thm}[Sikorav's Theorem for spaces]
	\label{Sikorav}
Let $X$ be a connected CW complex with finite $n$-skeleton and $\pi_1X=G$.  Let $\varphi\in S(G)$ and let $k\leq n$.  The following are equivalent:
\begin{enumerate}
    \item $\varphi\in\Sigma^k(X;R)$;
    \item $H_i^G(\widetilde X;\nov{R}{G}{\varphi})=0$ for $i\leq k$.
\end{enumerate}
\end{thm}

\begin{thm}[Sikorav's Theorem for groups {\cite[Theorem~5.3]{Fisher2021}}]
Let $G$ be a group of type $\mathsf{FP}_n(R)$.  Let $\varphi\in S(G)$ and let $k\leq n$.  The following are equivalent:
\begin{enumerate}
    \item $\varphi\in\Sigma^k(G;R)$;
    \item $H_i(G;\nov{R}{G}{\varphi})=0$ for $i\leq k$.
\end{enumerate}
\end{thm}

\section{Passing vanishing of Novikov homology to a division-closed subring}\label{sec passing vanishing}

In this section we will show that vanishing of Novikov homology up to a given degree can be determined by computing homology with coefficients being the divison closure of $RG$ in $\nov{R}{G}{\varphi}$.  First, we show that the division and rational closures are equal.

\begin{prop}\label{nov DivClos is RatClos}
Let $R$ be a ring and let $G$ be a finitely generated
 group.  If $\varphi\in S(G)$, then $\cald(RG\subset \nov{R}{G}{\varphi})=\calr(RG\subset\nov{R}{G}{\varphi})$.
\end{prop}
\begin{proof}
Let $\cald=\cald(RG\subset \nov{R}{G}{\varphi})$ and $\calr=\calr(RG\subset\nov{R}{G}{\varphi})$; clearly, $\cald \subseteq \calr$.  Consider a square matrix $A\in \mathbf{M}_n(\cald)$ that is invertible over $\nov{R}{G}{\varphi}$.  Let $B\in\mathbf{M}_n(\nov{R}{G}{\varphi} )$ be such that $AB=I$ where $I$ is the identity matrix.  We need to show that $A$ is invertible over $\cald$, which will show that $\cald$ is rationally closed, and hence that $\cald = \calr$.

Since $A$ is in particular a finite matrix over $\nov{R}{G}{\varphi}$, there exists $r \in \RR$ such that all entries of $A$ are supported on $\phi^{-1}\big( (-r, \infty) \big)$.
We truncate the entries of $B$ at $r$, and obtain a matrix $\bar B\in \mathbf{M}_n(RG)$.  Now, $A\bar B=I-P$ where the elements of the supports of the entries of $P$ are all positive with respect to $\varphi$.  Indeed, let $Q= B-\bar B$.  Then we have
\[I=AB=A(Q+\bar B)=AQ+A\bar B.\]
In particular, $A\bar B=I-AQ=I-P$.  Moreover, since $A\bar B\in\mathbf{M}_n(\cald)$ we have $P\in\mathbf{M}_n(\cald)$.

\begin{claim}\label{claim I plus P}
$I-P$ is invertible over $\cald$.
\end{claim}
\begin{claimproof}[Proof of \Cref{claim I plus P}]
The first diagonal entry of $I-P$ is of the form $1-u \in \cald$,	
 where every element in the support of $u$ has positive value under $\varphi$.  In particular, $1-u$ is invertible in $\nov{R}{G}{\varphi}$, and hence in $\cald$.  It follows that we may multiply $I-P$ with a diagonal matrix $D$ over $\cald$, that is invertible over $\cald$, to obtain $I-P'$ such that the first diagonal entry of $P'$ is $0$ and the remaining entries are either zero or have strictly positive support with respect to $\varphi$, in the same sense as $P$ did.

The entries of $I-P'$ all lie in $\cald$. We may now use elementary matrices over $\cald$ to clear the non-diagonal entries in the first row and column of $I - P'$. Repeating the process finitely many times for each diagonal entry of $I-P$ constructs a matrix $(I-P)^{-1} \in\mathbf{M}_n(\cald)$.
\end{claimproof}
To finish the proof we observe that $A\bar B(I-P)^{-1}=I$ and that $\bar B(I-P)^{-1}\in\mathbf{M}_n(\cald)$ since $\bar B\in\mathbf{M}_n(RG)$ and $(I-P)^{-1}\in\mathbf{M}_n(\cald)$ by \Cref{claim I plus  P}.
\end{proof}

\begin{prop}\label{inductive chain contraction}
Let $R$ be a ring, let $G$ be a finitely generated group, and let $\varphi\in S(G)$. Let
\[C_{n+1}\overset{\partial_{n+1}}{\to} C_n\overset{\partial_n}{\to} C_{n-1}\]
 be a chain complex of based free $\nov{R}{G}{\varphi}$-modules with $C_n$ and $C_{n-1}$ finitely generated and $\partial_i$ defined over $\cald\coloneqq \cald(RG\subset \nov{R}{G}{\varphi})$.  Suppose that we are given $\nov{R}{G}{\varphi}$-module homomorphisms $H_{n-1}\colon C_{n-1}\to C_n$ defined over $\cald$ and $H_n\colon C_n\to C_{n+1}$ defined over $\nov{R}{G}{\varphi}$ satisfying $H_{n-1}\partial_n+\partial_{n+1}H_n=\id_{C_n}$.  Then, we may take $H_{n}$ to be defined over $\cald$.
\end{prop}
\begin{proof}
We have the following diagram (which does not in general commute since the vertical map is equal to the sum of the two diagonal composites with codomain $C_n$):
\[\begin{tikzcd}
C_{n+1} \arrow[r, "\partial_{n+1}"] & C_n \arrow[r, "\partial_n"] \arrow[ld, "H_n"'] \arrow[d, "id_{C_n}"] & C_{n-1} \arrow[ld, "H_{n-1}"] \\
C_{n+1} \arrow[r, "\partial_{n+1}"] & C_n \arrow[r, "\partial_n"]                                           & C_{n-1}.
\end{tikzcd} \]

We have
\[
\partial_n = \partial_n (\partial_{n+1} H_n +  H_{n-1} \partial_n) = \partial_{n} H_{n-1} \partial_n
\]
and similarly $\partial_{n+1} = \partial_{n+1} H_n \partial_{n+1}$.
Using these two equalities we obtain
\begin{align*}
H_{n-1}\partial_n+\partial_{n+1}H_n &= \id_{C_n} \\
									&= \id_{C_n}\circ \id_{C_n}  \\
									&= (H_{n-1}\partial_n+\partial_{n+1}H_n)^2 \\
									&= H_{n-1}\partial_n+\partial_{n+1}H_n + \partial_{n+1} H_n H_{n-1} \partial_n
\end{align*}
and so $\partial_{n+1} H_n H_{n-1} \partial_n = 0$.

Since the modules are based, $H_n$ is naturally a matrix over $\nov{R}{G}{\varphi}$ with potentially infinitely many rows, finitely many columns, and finitely many nonzero entries. We will truncate the entries to obtain a matrix $\bar H_n$ over $RG$, and then construct a matrix $A$ such that  $\partial_{n+1} \bar H_{n}A+H_{n-1}\partial_{n}=\id_{C_n}$ and $\bar H_{n}A$ is defined over
\[\calr\coloneqq\calr(RG\subset \nov{R}{G}{\varphi}).\]
  From here we may apply \Cref{nov DivClos is RatClos}.  Of course nothing in life is that straight-forward so we spell out the details.

\begin{claim}\label{inductive chain contraction claim 1}
We have $C_n=\ker\partial_n\oplus\im H_{n-1}\partial_{n}$.
\end{claim}
\begin{claimproof}[Proof of \Cref{inductive chain contraction claim 1}]
 Consider $w\in\ker\partial_n \cap \im H_{n-1}\partial_{n}$. We have
\[
w = (H_{n-1}\partial_n+\partial_{n+1}H_n)w=\partial_{n+1}H_n w =0,
\]
where the second equality follows from $\partial_n w = 0$, and the third from
\[\partial_{n+1} H_n H_{n-1} \partial_n=0\]
 and $w \in \im H_{n-1} \partial_n$.
This shows that the sum in the claim is direct.

 Note that by the equation $\id_{C_n}=\partial_{n+1}H_n +H_{n-1}\partial_n $, it suffices to argue that $\im\partial_{n+1}H_n=\ker\partial_n$.  We have $\im\partial_{n+1}H_n\leq \ker\partial_n$.  For the other inclusion let $v\in\ker\partial_n$.  Then,
\[v=\partial_{n+1}H_n v +H_{n-1}\partial_n v = \partial_{n+1}H_n v, \]
which completes the proof of the claim.
\end{claimproof}

Note that $H_n$, when viewed as a matrix, has finitely many columns, and only finitely many non-zero entries (even though the number of rows might be infinite).
We truncate the entries of $H_n$ at a very large real number $r$ with respect to $\varphi$ to obtain a matrix $\bar H_n$ over $RG$.  We set
\[P = \partial_{n+1}(H_n-\bar H_n).\]
By choosing $r$ to be sufficiently large,  we arrange for the minimum of the support of every entry of $P$ with respect to $\varphi$ to be strictly greater than $0$.

\begin{claim}\label{inductive chain contraction claim 2}
	We have
$P(C_n)\leq \ker\partial_n.$
\end{claim}
\begin{claimproof}[Proof of \Cref{inductive chain contraction claim 2}]
First, suppose that $v\in\ker\partial_n$.  Then, \[v=H_{n-1}\partial_n v+\partial_{n+1}H_n v=\partial_{n+1} H_n v =Pv+\partial_{n+1} \bar H_n v.\]
So, $Pv=v-\partial_{n+1}\bar H_n v$ where $\partial_{n+1}\bar H_n v\in \im\partial_{n+1}\leqslant \ker\partial_n$.  In particular, $Pv\in\ker\partial_n$.

Now, suppose that $v\in \im H_{n-1}\partial_{n}$.  Then,
\[Pv=\partial_{n+1}H_nv-\partial_{n+1}\bar H_n v=-\partial_{n+1}\bar H_n v \]
which is contained in $\im\partial_{n+1} \leqslant \ker\partial_n$.  We are now done by applying the direct sum decomposition of \Cref{inductive chain contraction claim 1}.
\end{claimproof}

Write $I\coloneqq \id_{C_n}$.  The map $I-P\colon C_n\to C_n$ is invertible over $\nov{R}{G}{\varphi}$ with inverse $(I-P)^{-1}=I+\sum_{i=1}^\infty P^i$.  Moreover, we have
\begin{equation} \label{eqn inductive chain contraction}
H_{n-1}\partial_n(I-P)^{-1}=H_{n-1}\partial_n,
\end{equation}
because $H_{n-1}\partial_n P^i=0$ for $i\geq 1$ by \Cref{inductive chain contraction claim 2}.  We have
\[I-P= \partial_{n+1}\bar H_n+H_{n-1}\partial_n, \]
the right hand side of which is defined over $\cald$.  We  apply \Cref{nov DivClos is RatClos} to see that $I-P$ is invertible over $\cald$.  Finally,
\[I=(\partial_{n+1}\bar H_n+ H_{n-1}\partial_n)(I-P)^{-1}=\partial_{n+1}\bar H_n(I-P)^{-1}+ H_{n-1}\partial_n, \]
where the second equality uses \eqref{eqn inductive chain contraction}. Thus, we may take our new $H_n$ to be $\bar H_n(I-P)^{-1}$ which is defined over $\cald$.
\end{proof}

We are now ready to prove the main result of this section.

\begin{prop}\label{chain contractions}
Let $G$ be a finitely generated group and let $\varphi\in S(G)$.  Let $C_\bullet$ be a chain complex of free $RG$-modules finitely generated up to dimension $n$, with $C_k=0$ for $k<0$.  Then,
\[H_j(C_\bullet \otimes_{RG} \nov{R}{G}{\varphi})=0 \text{ for } j\leq n\]
if and only if
\[H_j\big(C_\bullet \otimes_{RG} \cald(RG\subset\nov{R}{G}{\varphi})\big)=0 \text{ for } j\leq n.\]
\end{prop}
\begin{proof}
First, suppose that $H_j(C_\bullet \otimes_{RG} \nov{R}{G}{\varphi})=0$ for $j\leq n$.  Let
\[H_i\colon C_i \otimes_{RG} \nov{R}{G}{\varphi} \to C_{i+1} \otimes_{RG} \nov{R}{G}{\varphi}\]
 be a chain homotopy over $\nov{R}{G}{\varphi}$ between $0$ and $\id_{C_\bullet}$ for $i\leq n$.

Let $\cald\coloneqq \cald(RG\subset\nov{R}{G}{\varphi})$.  We now inductively apply \Cref{inductive chain contraction} to $H$ and parts of the resolution $C_\bullet$,  with the base case clearly given by the chain complex $C_0\to0\to0$, to obtain a chain homotopy between $0$ and $\id_{C_\bullet}$ defined over $\cald$ for $i\leq n$.

Consider the chain complex $C_\bullet \otimes_{RG}  \cald$.  We have that \[ C_\bullet \otimes_{RG}  \cald \subset  C_\bullet \otimes_{RG}  \nov R G \varphi \]
and the larger chain complex admits a chain homotopy between $0$ and $\id_{C_n}$ in degrees less than or equal to $n$ defined over $\cald$.  In particular, $H_j( C_\bullet \otimes_{RG}  \cald)=0$ for $j\leq n$.

The converse follows easily by arguing with chain contractions.
\end{proof}

\section{Approximate Ore condition}\label{sec ring hom}

Throughout this section, $G$ denotes a group endowed with an epimorphism $\phi \colon G \to \Z$, with $K = \ker \phi$ and $t \in G$ such that $\phi(t)$ generates $\im \phi$. Various polynomials and power series with $t$ are always twisted, and the action of $t$ is always the conjugation action inside $\Z G,$ $\caln G$, and $\calu G$, depending on context.

\begin{defn}
Given a non-zero $x = \sum_{i=k}^\infty t^i x_i \in \calu K[t^{\pm 1} \rrbracket$ with $x_i \in \calu K$ for all $i$, and with $x_k \neq 0$, we define its \emph{initial term} $\init x$ to be $t^k x_k$. Note that $k$ might be negative. We will refer to $x_k$ as the \emph{pure part} of $\init x$, and to $k$ as the \emph{associated power}.
 We also set $\init 0 = 0$, with pure part $0$ and associated power $0$.
	
The \emph{nullity} $\nul x$  is defined to be the ${\caln K}$-dimension of the kernel of the pure part of $\init x$. 	
\end{defn}

The kernel of an affiliated operator in $\calu K$ is always a closed subspace of $\ell^2 K$, since the operator is closed, and hence it makes sense to talk about the $\caln K$-dimension of the kernel.

The definitions in particular apply to Laurent polynomials in $\calu K [t^{\pm 1}]$ and in $\caln K [t^{\pm 1}]$.

\begin{defn}	
	We say that a sequence $(q_n)_n$ in $\calu K[t^{\pm 1}\rrbracket$ is \emph{admissible} if the sequence of powers associated to $\init q_n$ is bounded from below.
	
	The sequence is
	 \emph{asymptotically injective} if
	%$\lim_n \nul q_n =0$, and \emph{strongly asymptotically injective} \iff
	it is admissible and
	\[\sum_n \nul q_n < \infty.\]
\end{defn}	

\begin{remark}
It is not hard to see that a term-wise product of asymptotically injective sequences is also asymptotically injective. 	
\end{remark}

\begin{lemma}
	\label{Linnell cor}
	For every $x \in \calu K [t^{\pm 1}]$ we have
	\[
	\dim_{\caln G} \ker x \leqslant \nul x.
	\]
\end{lemma}
\begin{proof}
	The result is clear when $x=0$. Let us assume it is not.
	
	Without loss of generality we may assume that the power associated to $\init x$ is $0$. Since $\calu K$ is the Ore localisation of $\caln K$, there exist an injective operator $z \in \caln K$ such that $zx \in \caln G$ and $\ker zx = \ker x$.
	Let $V = \ker x$. The polar decomposition gives us a partial isometry
	$v \in \caln K$ such that $\ker v = \ker \init zx$ and $\im v = \overline{\im \init zx}$.	Taking $u = 1-v$ we get a partial isometry $u$ with
$\ker u = (\ker \init zx)^\perp$ and $\im u = (\im \init zx)^\perp$.
	Then
	\[\init (u+zx) = u + \init zx\]
	is injective, and therefore $u + zx$ is injective by \cref{Linnell zero divs}. This means that $u\vert_V$ is injective as well, and hence $\dim_{\caln G} V \leqslant \dim_{\caln G} \im u$. Now, the latter dimension is equal to the $\caln G$-trace of $uu^*$. Since $uu^* \in \caln K$, this is equal to the $\caln K$-trace, and hence to $\dim_{\caln K} \im u = \nul zx = \nul x$.
\end{proof}

\begin{remark}
	\label{dim equality}
The proof above also shows that if $x \in \calu K$ then $\dim_{\caln G} \ker x = \dim_{\caln K} \ker x$.	
\end{remark}

\subsection{Approximate Ore condition}

We now introduce the main technical tool of this section.

\begin{prop}[Approximate Ore condition]
	\label{approx Ore}
	For every $q,q' \in \caln K[t^{\pm 1}]$ and every $\epsilon>0$  there exist $r,r' \in \caln K[t^{\pm 1}]$ such that $\nul r < \nul q' +  \epsilon$ and $\nul r' < \nul q + \nul q' +  \epsilon$, and
	\[
	qr = q'r'.
	\]
\end{prop}
\begin{proof}
	The proof is inspired by Tamari's argument \cite{Tamari1957}.
	
	It is clear that we may assume that the powers associated to $\init q$ and $\init q'$ are both $0$. Let $N$ denote the maximum of the degrees of $q$ and $q'$. We write
	\[
	q = \sum_{i=0}^N t^i q_i, \qquad q' = \sum_{i=0}^N t^i q'_i
	\]
	with $q_i, q_i' \in \caln K$.
	 For a natural number $k$, consider the right $\caln K$-linear map
	\begin{align*}
	\lambda_k \colon \caln K^{2k} &\to \caln K^{k+N} \\
	(x_0, y_0, \dots, x_{k-1}, y_{k-1}) &\mapsto \left( \sum_{i+j = l} t^{-j}(q_i t^j x_j -q'_i t^j y_j) \right)_{l}.
	\end{align*}
 (Secretly, we think of $(x_0, y_0, \dots, x_{k-1}, y_{k-1})$ as representing two twisted polynomials $r = \sum_{i=0}^{k-1} t^i x_i$ and $r' = \sum_{i=0}^{k-1} t^i y_i$, and lying in $\ker \lambda_k$ translates directly into $qr = q'r'$, since the right-hand side above collects the terms of $qr - q'r'$ according to the power of $t$.)

Let $d_k  = \dim_{\caln K} \ker \lambda_k$, and note that $d_k \geqslant k - N$.
We may embed $\caln K^{2k} \to \caln K^{2+2k}$ and $\caln K^{k + N} \to \caln K^{1 + k +N}$ by augmenting vectors with zeroes in initial positions. These embeddings form commutative squares with the maps $\lambda_k$ and $\lambda_{k+1}$.
The image of $\ker \lambda_k$ under the first map will be denoted by $t \ker \lambda_k$; it is a subspace of $\ker \lambda_{k+1}$ of dimension $d_k$.%  The two subspaces intersect in $t \ker \lambda_{k-1}$ (for $k\geqslant 2$).

Let $p, p' \colon \caln K^{2k} \to \caln K$ denote the projections onto, respectively, the first and the second factor.

\begin{claim}
	For some $k$ we have
	\[
	\dim_{\caln K} p(\ker \lambda_k) > 1-\nul q' - \epsilon.
	\]
\end{claim}
\begin{claimproof}[Proof of claim]
Without loss of generality, we will assume that $\epsilon<1$.	
	
Consider $(x_0, y_0, \dots, x_{k-1}, y_{k-1}) \in  \ker p\vert_{\ker \lambda_k}$. We immediately see that $x_0 = 0$, and $q'_0y_0=0$. Hence $p'$ sends $\ker p\vert_{\ker \lambda_k}$ to $\ker q'_0$. Therefore,
\[
\dim_{\caln K} \left(\ker p\vert_{\ker \lambda_k}\right) - \nul q' \leqslant \dim_{\caln K} \left( \ker p\vert_{\ker \lambda_k} \cap \ker p'\vert_{\ker \lambda_k}\right).
\]
The intersection of the kernels on the right-hand side is precisely $t \ker \lambda_{k-1}$. Hence,
\[
d_k - \dim_{\caln K} p( \ker \lambda_k) - \nul q' \leqslant d_{k-1}.
\]
Rearranging, we obtain
\[
\dim_{\caln K} p( \ker \lambda_k) \geqslant d_{k} - d_{k-1} - \nul q'.
\]
If the left-hand side is bounded above by $1-\nul q' - \epsilon$ for all $k$, then
\[
1 - \epsilon \geqslant  d_{k} - d_{k-1}
\]
and so adding such  terms together gives
\[
(1-\epsilon)k \geqslant d_k \geqslant k - N
\]
for all $k$, which is a contradiction. We conclude that for some $k$ the $\caln K$-dimension of $p(\ker \lambda_k)$ is greater than $1-\nul q' - \epsilon$, as claimed.
\end{claimproof}

It follows that for $k$ as above, $p(\ker \lambda_k)$ is dense inside of $\pi_V \caln K$ for some closed subspace $V \leqslant \ell^2(K)$ of $\caln K$-dimension greater than $1-\nul q' - \epsilon$. Take $x \in \ker \lambda_k$ such that $\pi_V - p(x)$ has norm less than $1$. Then
\[V \cap \ker p(x) = \{0\},\]
 and so $\ker p(x)$ has dimension less than $\nul q' + \epsilon$.
Write
\[x = (x_0, y_0, \dots, x_{k-1}, y_{k-1});\]
 let $r = \sum_{i=0}^{k-1} t^i x_i$ and $r' = \sum_{i=0}^{k-1} t^i y_i$. We have shown that  $\ker p(x) = \ker x_0$ has dimension less than $\nul q' + \epsilon$. Also, $x \in \ker \lambda_k$ means precisely that
\[
qr = q'r'.
\]
Finally, the last equality implies that $q_0 \init r = q'_0 \init {r'} $, and hence
\[\nul r' \leqslant \nul q + \nul r < \nul q + \nul q' +  \epsilon. \qedhere \]
\end{proof}

\begin{remark}
	\label{caln to calu}
	Clearly, if $q$ and $q'$ are Laurent polynomials over $\calu K$, we may multiply them by a suitable injective operator in $\caln K$ and obtain Laurent polynomials over $\caln K$ with the same nullities. Hence the conclusion of the above statement holds verbatim if $q$ and $q'$ lie in $\calu K[t^{\pm 1}]$, and we will use it in this generality.
\end{remark}

\begin{corollary}
	\label{common multi}
	Every two asymptotically injective sequences $(q_n)_n$ and $(q'_n)_n$ over $\calu K[t^{\pm 1}]$ admit an \emph{asymptotically injective common multiple}, that is, an asymptotically injective sequence $(x_n)_n$ over $\caln K[t^{\pm 1}]$ such that there exist two asymptotically injective sequences $(y_n)_n$ and $(z_n)_n$ over $\caln K[t^{\pm 1}]$ with $x_n = q_n y_n = q'_n z_n$ for every $n$.
\end{corollary}
\begin{proof}
	For every $n$ we obtain $x_n, y_n$, and $z_n$ with $x_n = q_n y_n = q'_n z_n$ from \cref{approx Ore}, setting $\epsilon = 2^{-n}$. This way
	\[
	\sum_n \nul x_n \leqslant \sum_n \left( \nul q_n + \nul y_n  \right) \leqslant \sum_n \left( \nul q_n + \nul q'_n + 2^{-n} \right) < \infty
	\]
and similarly for $(y_n)_n$ and $(z_n)_n$.
\end{proof}

\subsection{Asymptotic agreement}

In this section we are dealing with sequences, but in reality we think of them as proxies, and we are really interested in their limits (which we will define later). Hence it is natural to introduce an equivalence relation on sequences.

\begin{defn}
	We say that two sequences $(x_n)_n$ and $(y_n)_n$ in $\calu G$ \emph{asymptotically agree as operators}, written $(x_n)_n \apa (y_n)_n$,  if
	\[
	\sum_n \dim_{\caln G} (\ker(x_n - y_n)^\perp) < \infty.
	\]

A sequence $(x_n)_n$ in $\calu G$ \emph{stabilises} if $(x_n)_n \apa(x_{n+1})_n$.
\end{defn}

Note that if $(x_n)_n$ and $(y_n)_n$ are sequences in $\calu K$, then
\[\dim_{\caln G} \ker(x_n - y_n) = \dim_{\caln K} \ker(x_n - y_n)\]
 by \cref{dim equality}, and hence such sequences asymptotically agree as sequences in $\calu K$ \iff they asymptotically agree as sequences in $\calu G$. Therefore, there is no need to specify over which group we are working.

It is very easy to see that being in asymptotic agreement is an equivalence relation.

\begin{lemma}
\label{apa}
Let $(x_n)_n$, $(y_n)_n$, and $(z_n)_n$ be sequences in $\calu G$. If $(x_n)_n \apa (y_n)_n$, then all of the following hold:
\begin{enumerate}
	\item \label{apa 1}$(x_n + z_n)_n \apa (y_n + z_n)_n$,
	\item \label{apa 2}$(x_nz_n)_n \apa (y_nz_n)_n$,
	\item \label{apa 3}$(z_nx_n)_n \apa (z_ny_n)_n$.
\end{enumerate}	
\end{lemma}	
\begin{proof}
	This is immediate. For \eqref{apa 2} and \eqref{apa 3}, it is enough to observe that the dimension of the kernel of a product of operators is bounded from below by the dimension of the kernel of either factor.
\end{proof}

We now extend the relation $\apa$ to power series.
\begin{defn}
	Two sequences $(x_n)_n$ and $(y_n)_n$ over $\calu K [ t^{\pm 1} \rrbracket$ \emph{asymptotically agree as power series}, written $(x_n)_n \apa_K (y_n)_n$, if %for some degree the Laurent power series $x_n$ all have coefficients zero below that degree, and
	for every fixed degree $d$, the sequence (over $\calu K$) of coefficients of $x_n$ by $t^d$ and the sequence of coefficients of $y_n$ by $t^d$ asymptotically agree as operators.
	
	A sequence $(x_n)_n$  over $\calu K [ t^{\pm 1} \rrbracket$ is \emph{$K$-stabilising} if $(x_n)_n \apa_K (x_{n+1})_n$.
\end{defn}

For power series, we add $K$ as a subscript to $\apa$ due to the potential confusion for sequences of Laurent polynomials in $\calu K[t^{\pm 1}]$. Such Laurent polynomials are at the same time elements of $\calu G$, in which case the definition of $\apa$ applies, and Laurent power series, in which context we use $\apa_K$. It is clear that  $\apa_K$ is again an equivalence relation.

We now collect basic arithmetic properties of the equivalence relation $\apa_K$.

\begin{lemma}
	\label{apaK}
	Let $(x_n)_n$, $(y_n)_n$, and $(z_n)_n$ be sequences in $\calu K [t^{\pm 1} \rrbracket$. If $(x_n)_n \apa_K (y_n)_n$, then all of the following hold:
	\begin{enumerate}
		\item \label{apaK 1}$(x_n + z_n)_n \apa_K (y_n + z_n)_n$,
		\item \label{apaK 2}$(x_nz_n)_n \apa_K (y_nz_n)_n$,
		\item \label{apaK 3}$(z_nx_n)_n \apa_K (z_ny_n)_n$.
	\end{enumerate}
\end{lemma}
\begin{proof}
	This follows from \cref{apa}.
\end{proof}

\subsection{Partial inverse}

We are now approaching the main technical onslaught. We will introduce two constructions that play the role of inverses of elements in $\calu K[t^{\pm 1}]$, one in $\calu G$, and one in $\calu K[ t^{\pm 1} \rrbracket$. They will share many properties.

Recall that in $\calu G$ we have the notion of a partial inverse $x \mapsto x^\dagger$.
The partial inverse of an element $x \in \calu K$ lies in $\calu K$, and for such an $x$ we have $(t^i x)^\dagger =  x^\dagger t^{-i}$.

\begin{lemma}
	\label{agreement}
	Let $(p_n)_n$ be a sequence in $\calu K [t^{\pm 1}]$,  and let $(s_n)_n$ and $({q_n})_n$ be asymptotically injective sequences in $\calu K [t^{\pm 1}]$. All of the following hold:
	\begin{enumerate}
		\item \label{agreement 1} $\left( (q_n s_n)^\dagger \right)_n \apa \left( {s_n}^\dagger {q_n}^\dagger \right)_n$.
		\item \label{agreement 2} $(s_n {s_n}^\dagger)_n \apa (1)_n$.
		\item \label{agreement 3} $(p_n {q_n}^\dagger)_n \apa (p_n s_n (q_ns_n)^\dagger)_n$.
	\end{enumerate}
\end{lemma}
\begin{proof} We prove the items in turn.
	\begin{enumerate}
		\item 	We have
		\[
		(q_ns_n)(q_ns_n)^\dagger - q_n s_n {s_n}^\dagger {q_n}^\dagger= \pi_{\overline{\im q_n s_n}} - q_n \pi_{\overline{\im s_n}} {q_n}^\dagger.
		\]
		The right-hand side is $0$ on $(\im q_n)^\perp$, since both summands are $0$ there, and on $q_n(\im s_n \cap (\ker q_n)^\perp)$, since the summands restrict to the identity there. \cref{Linnell cor} tells us that $\dim_{\caln G} \ker s_n \leqslant \nul s_n$. Hence
		\[
		\dim_{\caln G} \ker ( (q_ns_n)(q_ns_n)^\dagger - q_n s_n {s_n}^\dagger {q_n}^\dagger ) \geqslant 1- \nul s_n.
		\]
		Since, using the same argument as above,
		\[
		\dim_{\caln G} \ker q_ns_n \leqslant \nul q_n + \nul s_n,
		\]
		we conclude that
		\[
		\dim_{\caln G} \ker ( (q_ns_n)^\dagger - {s_n}^\dagger {q_n}^\dagger ) \geqslant 1- 2\nul s_n - \nul q_n.
		\]
		Thus,
		\[
		\sum_n \dim_{\caln G} \big( \ker ( (q_ns_n)^\dagger - {s_n}^\dagger {q_n}^\dagger )^\perp\big) \leqslant \sum_n ( 2\nul s_n + \nul q_n) < \infty.
		\]
		\item The operator
		\[
		1 - s_n {s_n}^\dagger   = 1 - \pi_{\overline{\im s_n}}
		\]
		has kernel of dimension at least $1- \nul s_n$, and we finish the argument as above.
		\item This follows immediately from the two items above and \cref{apa}\eqref{apa 2} and \eqref{apa 3}. \qedhere
	\end{enumerate}
\end{proof}

%From the last item, we immediately obtain the following.

\begin{lemma}
	\label{adjoints and daggers}
	Let $(x_n)_n$ and $(y_n)_n$ be sequences in $\calu G$.
	If $(x_n)_n$ and $(y_n)_n$ asymptotically agree as operators, then so do the sequences of adjoints $({x_n}^*)_n$ and $({y_n}^*)_n$, and the sequences of partial inverses $({x_n}^\dagger)_n$ and $({y_n}^\dagger)_n$.
\end{lemma}
\begin{proof}
	This is clear for adjoints, since $\ker ({x_n}^* - {y_n}^*) = \left( \im (x_n - y_n) \right) ^\perp$ has the same $\caln G$-dimension as $\ker (x_n - y_n)$.
	
	For partial inverses, we need to introduce some notation. Let
	\[
	1-d_n = \dim_{\caln G} \ker( x_n - y_n),
	\]
	and note that $\sum_n d_n < \infty$.
	
	The subspace $\ker ( {x_n}^\dagger - {y_n}^\dagger)$ contains $(\im x_n)^\perp \cap (\im y_n)^\perp$, since this is the intersection of the kernels of ${x_n}^\dagger$ and ${y_n}^\dagger$, and it contains
	\[\overline{x_n( \ker (x_n - y_n) \cap (\ker x_n)^\perp \cap (\ker y_n)^\perp )},\]
	 since this is the closure of a subspace on which ${x_n}^\dagger$ and ${y_n}^\dagger$ act as the identity -- here, we are using the fact that the action of $x_n$ and $y_n$ are the same on $\ker( x_n - y_n)$.
	  Observe that
	$ (\im x_n)^\perp \cap (\im y_n)^\perp$ and
	\[\overline{x_n( \ker (x_n - y_n) \cap (\ker x_n)^\perp \cap (\ker y_n)^\perp )}
	\]
	are perpendicular.
	 We will now bound the dimensions of these two spaces from below.
	
	We have
	\[\im x_n \geqslant x_n(\ker (x_n - y_n))  = y_n(\ker (x_n - y_n)) \leqslant \im y_n. \]
	Moreover, the definition of $d_n$ tells us that the $\caln G$-dimension of the perpendicular complement of  $x_n(\ker (x_n - y_n))$ in $\overline{\im x_n}$ is bounded above by $d_n$, and so is the complement in $\overline{\im y_n}$. We conclude that
	\begin{align*}
	\dim_{\caln G} \left( (\im x_n)^\perp \cap (\im y_n)^\perp \right) &\geqslant  \dim_{\caln G} \left( x_n(\ker (x_n - y_n)) \right)^\perp - 2d_n \\
	&\geqslant  \dim_{\caln G} (\im x_n)^\perp - 3d_n.
	\end{align*}
	
	We will now focus on $x_n( \ker (x_n - y_n) \cap (\ker x_n)^\perp \cap (\ker y_n)^\perp )$.
	We have
	\[	
	 (\ker x_n)^\perp \cap (\ker y_n)^\perp =  \im {x_n}^* \cap \im {y_n}^*
	\geqslant  {x_n}^*(\ker ({x_n}^* - {y_n}^*)) .
	\]
	The $\caln G$-codimension of this last subspace in $\im {x_n}^* = (\ker x_n)^\perp $ is bounded from above by $d_n$,
and therefore
	\begin{align*}
	&  \dim_{\caln G}  \overline{x_n( \ker (x_n - y_n) \cap (\ker x_n)^\perp \cap (\ker y_n)^\perp )} \\ &\geqslant 	\dim_{\caln G}  \overline{x_n( \ker (x_n - y_n) \cap (\ker x_n)^\perp )} -d_n\\
	&\geqslant 	\dim_{\caln G}  \overline{x_n( (\ker x_n)^\perp )} -2d_n\\
	&= \dim_{\caln G}  \overline{\im x_n} -2d_n.
	\end{align*}

	Combining the last two inequalities gives $\dim_{\caln G} \ker ( {x_n}^\dagger - {y_n}^\dagger) \geqslant 1 - 5d_n$, and hence $\dim_{\caln G}\ker(x_n^\dagger - y_n^\dagger)^\perp \leq 5d_n$, and the result follows.
\end{proof}

In particular, if $(x_n)_n$ stabilises, then so do $({x_n}^*)_n$ and $({x_n}^\dagger)_n$.

We will now introduce the first limit -- it will later allow us to view (equivalence classes of) our sequences as elements of $\calu G$.

\begin{lemma}
	\label{limit}
For every stabilising sequence $(x_n)_n$ over $\calu G$, there exists a unique element $x_\infty \in \calu G$ such that $(x_n)_n$ and $(x_\infty)_n$ asymptotically agree.
\end{lemma}
\begin{proof}
	Let $V_n = \ker (x_{n+1} - x_n)$. By assumption, $\sum_n (1-\dim_{\caln G} V_n) < \infty$. Let $U_n = \overline{\sum_{m\geqslant n} {V_m}^\perp}$. The subspaces $U_n$ are closed and $G$-invariant, and form a nested sequence with $\lim_n \dim_{\caln G} U_n = 0$. Therefore $\bigcap_n U_n = \{0\}$.
	
	Observe that the affiliated operators $x_m$ with $m \geqslant n$ all agree on $L_n$, where $L_n$ is defined to be ${U_n}^\perp$ intersected with all of their domains.
	In fact, $L_n$ is equal to ${U_n}^\perp$ intersected with the domain of $x_n$, since the domains are essentially dense and therefore their intersections with ${U_n}^\perp$ are dense therein, and since the operators are closed.
	Now, the subspaces $L_n$ are $G$-invariant, form an ascending chain, and $\overline{ L} = \ell^2 G$ with $L = \bigcup L_n$. We define $x'_\infty \colon L \to \ell^2 G$ by $x'_\infty\vert_{L_n} = x_n\vert_{L_n}$. It is clear that $x'_\infty$ is densely defined and $G$-equivariant.
	
	  We now apply the same procedure to the stabilising sequence $({x_n}^*)_n$, and obtain a densely defined $G$-invariant operator $x^*_\infty$. It is easy to see that $x^*_\infty$ is the adjoint of $x'_\infty$. Since $x'_\infty$ is densely defined, the adjoint $x_\infty  = ({x^*_\infty})^*$ is defined on a superspace of $L$, and on $L$ agrees with $x'_\infty$. Moreover, since $x^*_\infty$ is densely defined, $x_\infty$ is closed. Hence $x_\infty$ is the desired affiliated operator.
	
	  To prove uniqueness, suppose that we have another affiliated operator $y_\infty$ such that$(x_n)_n$ and $(y_\infty)_n$ asymptotically agree. Then the sequences $(x_\infty)_n$ and $(y_\infty)_n$ asymptotically agree, forcing $\ker (x_\infty - y_\infty) = \ell^2 G$. This means that $x_\infty = y_\infty$ as affiliated operators.
\end{proof}

We will refer to the element $x_\infty$ as the \emph{limit} of the sequence $(x_n)_n$.  The map $(x_n)_n \mapsto x_\infty$ will be denoted by $\lambda_{\calu G}$. Note that if $x_n \in \calu K$ for every $n$, then $x_\infty \in \calu K$ as well.

\subsection{Expansion}

We are now ready to construct the second function that will serve as an inverse, this time in $\calu K[t^{\pm 1} \rrbracket$.

\begin{defn}[Expansion]
	We define the \emph{expansion map}
	\[
	\calu K[t^{\pm 1}\rrbracket \smallsetminus \{0\} \to \calu K [ t^{\pm 1}\rrbracket
	\]
	 by
	\[
	q\mapsto \overline q =   (\init q)^\dagger \sum_{k=0}^\infty \left( ( \init q - q) (\init q)^\dagger \right)^k.
	\]
\end{defn}

In the unlikely event of the reader not recognising the construction immediately, it is instructive to consider the case in which $\init q$ is invertible in $\calu K$.

\begin{lemma}
	\label{ex for invertibles}
	For $q \in \calu K[t^{\pm 1}\rrbracket$, suppose that $\init q$ is invertible in $\calu K$. Then $\overline q$ is precisely the inverse of $q$ in $\calu K [ t^{\pm 1}\rrbracket$.
\end{lemma}
\begin{proof}
	Observe that $({{\init q}})^\dagger$ is the inverse of ${\init q}$.
	Let $r = q(\init q)^\dagger$, and observe that we have $r = 1 + \sum_{i>0} t^i r_i$ with $r_i \in \calu K$. Hence
	\[
	\overline r = \sum_{k \geqslant 0} (1-r)^k
	\]
	is the inverse of $r$. Thus, $(\init q)^\dagger \sum_{k \geqslant 0} (1-r)^k$ is the inverse of $q$. But
	\[
	 (\init q)^\dagger \sum_{k \geqslant 0} (1-r)^k =  (\init q)^\dagger \sum_{k \geqslant 0} ( \init q   (\init q)^\dagger -q (\init q)^\dagger)^k = \overline q. \qedhere
	\]
\end{proof}

\begin{lemma}
	\label{ex is an inverse}
	Let $(q_n)_n$ and $(r_n)_n$ be asymptotically injective sequences over $\calu K[t^{\pm1} \rrbracket$. All of the following hold:
	\begin{enumerate}
		\item If $(q_n)_n \apa_K (r_n)_n$ then  $(\overline{q_n})_n \apa_K (\overline{r_n})_n$;
				\item $(\overline{\overline{q_n}})_n \apa_K (q_n)_n$;
		\item $(q_n \overline{q_n})_n \apa_K (\overline{q_n}q_n)_n \apa_K (1)_n$;
		\item $(\overline{ q_n r_n})_n \apa_K (\overline{r_n} \cdot \overline{q_n} )_n$.
	\end{enumerate}
\end{lemma}
\begin{proof}
	It is easy to see what happens when we multiply the terms of our sequences by powers of $t$, and hence we will assume that $\init q_n$ and $\init r_n$ lie in $\calu K$ for all $n$.
	
	\begin{enumerate}
		\item For every fixed $d$, the terms of $\overline{q_n}$ and $\overline{r_{n}}$ appearing next to $t^d$ are obtained from finitely many corresponding terms in $q_n$ and $r_{n}$ via the same arithmetic operation. We now apply \cref{apa} for every degree $d$ separately.
		\item Recall that the polar decomposition gives us partial isometries $u_n$ and $v_n$ in $\caln K$ mapping $\ker \init q_n$ onto $(\im \init q_n)^\perp$ and $\ker \init r_n$ onto $(\im \init r_n)^\perp$ and being trivial on $(\ker \init q_n)^\perp$ and $(\ker \init r_n)^\perp$, respectively. The operators $\init q_n + u_n$ and $\init r_n + v_n$ are then invertible in $\calu K$.
		Crucially, $(u_n)_n \apa_K (0)_n \apa_K (v_n )_n$, since $(q_n)_n$ and $(r_n)$ are asymptotically injective.  \cref{apaK} tells us that \[(q_n + u_n)_n \apa_K (q_n)_n \quad \textrm{and} \quad (r_n + v_n)_n \apa_K (r_n)_n.\]
		Now, \cref{ex for invertibles} yields $\overline{\overline{q_n + u_n}} = q_n + u_n$ for all $n$, and combining this with the previous item finishes this part of the proof.
		\item	
		By \cref{ex for invertibles}, for all $n$ we have
		\[
		(q_n + u_n)\overline{(q_n + u_n)} = \overline{(q_n + u_n)}(q_n + u_n) = 1.
		\]
		We are now done thanks to the first item and \cref{apaK}.
		\item Finally, by \cref{ex for invertibles} and uniqueness of inverses
		we have
		\[
		\overline{(q_n + u_n)(r_n + v_n)} = \overline{r_n + v_n}\cdot \overline{q_n + u_n}.
		\]
By \cref{apaK}, $(q_nr_n)_n \apa_K  (q_n + u_n)(r_n + v_n)_n$, and we finish the proof using the first item. \qedhere
	\end{enumerate}	
\end{proof}

%Given two affiliated operators $x, y \in calu G$ and a subspace $V \leqslant \ell^2 G$, we will write $x|_V = y|_V$ to mean that $x$ and $y$ agree when restricted to the intersection of their domains and $V$.
	
\subsection{Weakly rational elements}
	
\begin{defn}
	Let $\prewrat (K,t)$ be the set of sequences $(p_n,q_n)_n$ such that all of the following hold:
	\begin{itemize}
		\item $(p_n)_n$ is an admissible sequence in $\calu K[t^{\pm 1}]$,
		\item $(q_n)_n$ is an asymptotically injective sequence in $\calu K [t^{\pm 1}]$,
		\item the sequence $(p_n (q_n)^\dagger)_n$ stabilises,
		\item the sequence $(p_n \overline{q_n})_n$ $K$-stabilises.
	\end{itemize}

We let $\sim$ be a relation on $\prewrat (K,t)$ given by $(p_n,q_n)_n \sim (p'_n,q'_n)_n$ if there exist asymptotically injective sequences $(r_n)_n$ and $(s_n)_n$ in $\calu K[t^{\pm 1}]$ such that
\[
p_n r_n = p'_ns_n, \qquad q_nr_n = q'_ns_n
\]
for all $n$.
\end{defn}

\begin{lemma}
	The relation $\sim$ is an equivalence relation.
\end{lemma}
\begin{proof}
	Reflexivity is clear by taking $r_n = 1 = s_n$ for every $n$. Symmetry is also clear by exchanging the sequences  $(r_n)_n$ and $(s_n)_n$. Transitivity follows from \cref{common multi}. Here are the details.
		
	Consider three sequences $(p_n,q_n)_n, (p'_n,q'_n)_n$ and $(p''_n,q''_n)_n$ in $\prewrat (K,t)$, and suppose that $(p_n,q_n)_n \sim (p'_n,q'_n)_n$ and $(p'_n,q'_n)_n \sim (p''_n,q''_n)_n$. The definition gives us asymptotically injective sequences $(r_n)_n, (s_n)_n, (r'_n)_n$, and $(s'_n)_n$ such that
\begin{align*}
p_n r_n = p'_ns_n, \hspace{5mm} & q_nr_n = q'_ns_n, \\
 p'_n r'_n = p''_ns'_n, \hspace{5mm} & q'_nr'_n = q''_ns'_n,
\end{align*}
for all $n$.
By \cref{common multi}, there exists an asymptotically injective  right common multiple $(x_n)_n$ of $(s_n)_n$ and $(r'_n)_n$.
Let $y_n$ and $z_n$ be such that $x_n = s_n y_n = r'_n z_n$. Set $r''_n = r_n y_n$ and $s''_n = s'_n z_n$, and note that $(r''_n)_n$ is asymptotically injective. Then
\[
p_n r''_n = p_n r_n y_n = p'_n s_n y_n = p'_n x_n = p'_n r'_n z_n = p''_n s'_n z_n = p''_n s''_n. \qedhere
\]
%Moreover, $\nul r''_n \leqslant \nul r_n + \nul y_n \leqslant \nul r_n + \nul x_n$, which implies that $(r''_n)_n$ is asymptotically injective. An analogous argument shows that $(s''_n)_n$ is asymptotically injective.
\end{proof}

Let $\wrat (K,t)$ denote the set of equivalence classes in $\prewrat (K,t)$ under this equivalence relation. We will abuse notation by not differentiating between the elements of $\prewrat (K,t)$ and the equivalence classes they lie in. The elements of $\wrat (K,t)$ will be called \emph{weakly rational sequences}.

\begin{lemma}
	\label{homothety}
	For every $(p_n,q_n)_n \in \prewrat (K,t)$ and every asymptotically injective sequence $(x_n)_n$ in $\calu K[t^{\pm 1}]$, the sequence $(p_nx_n,q_nx_n)_n$ lies in $\prewrat (K,t)$ in the same equivalence class as $(p_n,q_n)_n$.
\end{lemma}
\begin{proof}
	The only conditions that are non-trivial to check are the stability of $(p_n x_n (q_n x_n)^\dagger)_n$ and $(p_n x_n \overline{q_n x_n})_n$. For the first sequence, this follows immediately from \cref{agreement}\eqref{agreement 3}; for the second, we use \cref{ex is an inverse,apaK}.
\end{proof}

\begin{prop}
	\label{ring struct}
	The map $\iota \colon p \mapsto (p,1)_n$ embeds $\calu K[t^{\pm 1}]$ into $\wrat(K,t)$, and $\wrat(K,t)$ admits a ring structure making this embedding into a ring homomorphism.
\end{prop}
\begin{proof}
	We will break the proof into three parts.

	\begin{markproof}[Embedding]
 We start with the first claim. Suppose that $(p,1)_n \sim (p',1)_n$. The definition gives us an asymptotically injective sequence $(r_n)_n$ with $pr_n = p'r_n$.
	Let $x = \init (p-p')$. We then have
	\[
	x \init r_{n} = 0
	\]
	for all $n$. Since $\lim_n \nul r_n =0$, the images of the operators $\init r_{n}$ have closures with dimensions tending to $1$. Hence $x = 0$, and so $p = p'$.\end{markproof}
	
	\begin{markproof}[Ring structure: addition]
	Now we need to define the ring structure. We start with addition. Let $(p_n,q_n)_n$ and $(p'_n, q'_n)_n$ in $\wrat(K,t)$ be given. We are first going to bring them to a common denominator: \cref{common multi} gives us asymptotically injective sequences $(x_n)_n, (y_n)_n$, and $(z_n)_n$ with $x_n = q_ny_n = q'_n z_n$ for every $n$. \cref{homothety} tells us that $(p_n y_n, x_n)_n$ and $(p'_n z_n, x_n)_n$ lie in $\wrat(K,t)$; we define their sum to be $(p_ny_n + p'_nz_n,x_n)_n$. We now need to check that this sequence lies in $\wrat(K,t)$, and that the sum is independent of the choices of representatives.
	
To check that $(p_ny_n + p'_nz_n,x_n)_n$ lies in $\wrat(K,t)$, we need to verify the stability conditions of the definition. This is easy, since
\[
(p_ny_n + p'_nz_n){x_n}^\dagger = p_ny_n{x_n}^\dagger + p'_nz_n{x_n}^\dagger,
\]
and
\[
(p_ny_n + p'_nz_n)\overline{x_n} = p_ny_n\overline{x_n}+ p'_nz_n\overline{x_n}.
\]

Now suppose that we have picked different sequences $(x'_n)_n, (y'_n)_n$, and $(z'_n)_n$ with properties analogous to the ones of $(x_n)_n, (y_n)_n$, and $(z_n)_n$. Thanks to \cref{common multi,homothety} we may assume that $(x_n)_n = (x'_n)_n$, which in particular implies that $q_n y_n = q_n y'_n$ and $q'_nz_n = q'_n z'_n$.
We now need to show that $(p_n y_n + p'_n z_n, x_n)_n \sim (p_n y'_n + p'_n z'_n, x_n)_n$.
\cref{approx Ore,caln to calu} give us sequences $(a_n)_n$ and $(b_n)_n$ in $\caln K[t^{\pm 1}]$ such that for all $n$
\[
(y_n - y'_n)a_n = \pi_{(\ker \init q_n)^\perp} b_n,
\]
and such that $(a_n)_n$ is asymptotically injective. We claim that
\[
\pi_{(\ker \init q_n)^\perp} b_n = 0.\]
 We have
\[
0 = 0 \cdot a_n  = q_n (y_n - y'_n)a_n = q_n \pi_{(\ker \init q_n)^\perp} b_n.
\]
Hence,
\[
\init q_n \init (\pi_{(\ker \init q_n)^\perp} b_n) = 0.
\]
But $\init (\pi_{(\ker \init q_n)^\perp} b_n)$ is of the form $\pi_{(\ker \init q_n)^\perp} c t^j$ for some $j$ and $c \in \caln K$, and $\init q_n$ is clearly injective on $(\ker \init q_n)^\perp$. Hence $\pi_{(\ker \init q_n)^\perp} c t^j = 0$ and the claim follows.
We deduce that $(y_n - y'_n) a_n = 0$.
We then have
\[(p_n y_n + p'_n z_n, x_n)_n \sim (p_n y_n a_n + p'_n z_n a_n, x_n a_n)_n\]
 and $p_n y_n a_n = p_n y'_n a_n$. We repeat the argument for $z_na_n$ and $z'_n a_n$. This shows that addition is well defined.

It is clear that addition has a neutral element $(0,1)_n$, and that
\[(p_n,q_n)_n + (-p_n, q_n)_n = (0,q_n)_n \sim (0,1)_n.\qedhere\]
\end{markproof}

\begin{markproof}[Ring structure: multiplication]
We are left with defining multiplication. Let $(p_n,q_n)_n$ and $(p'_n,q'_n)_n$ be given elements of $\wrat(K,t)$. Using \cref{approx Ore}, we get an asymptotically injective sequence $(s_n)_n$ and a sequence $(r_n)_n$ such that
\[
p'_n s_n = q_n r_n
\]
for every $n$.
We define $(p_n,q_n)_n \cdot (p'_n, q'_n)_n = (p_n r_n, q'_n s_n)$. We first need to check that the output is an element of $\wrat(K,t)$. By \cref{agreement}\eqref{agreement 1}, $\left( (q'_ns_n)^\dagger \right)_n \apa ({s_n}^\dagger {q'_n}^\dagger)_n$; by \eqref{agreement 2},
  $({q_n}^\dagger q_n)_n \apa ({s_n}^\dagger s_n)_n \apa (1)_n$. We therefore have
  \begin{align*}
  \left( p_n r_n (q'_n s_n)^\dagger \right)_n &\apa \left( p_n {q_n}^\dagger q_n r_n {s_n}^\dagger {q'_n}^\dagger \right)_n\\
  &=\ \ \left( p_n {q_n}^\dagger p'_n s_n {s_n}^\dagger {q'_n}^\dagger \right)_n\\
  &\apa \left( p_n {q_n}^\dagger p'_n  {q'_n}^\dagger \right)_n.\end{align*}
    Since $\left( p_n {q_n}^\dagger \right)_n$ and $\left( p'_n  {q'_n}^\dagger \right)_n$ stabilise,
   we have
    \begin{align*}
    \left( p_n {q_n}^\dagger p'_n  {q'_n}^\dagger \right)_n &\apa \left( p_n {q_n}^\dagger p'_{n+1}  {q'_{n+1}}^\dagger \right)_n\\
    &\apa \left( p_{n+1} {q_{n+1}}^\dagger p'_{n+1}  {q'_{n+1}}^\dagger \right)_n\\
    &\apa \left( p_{n+1} r_{n+1} (q'_{n+1} s_{n+1})^\dagger \right)_n.
    \end{align*}

 The argument for
 \[\left( p_n r_n \overline{q'_n s_n} \right)_n \apa_K \left( p_{n+1} r_{n+1} \overline{q'_{n+1} s_{n+1}} \right)_n\]
  is analogous.

The next step is to check that our multiplication is well defined. There are three parts to this. First, for fixed sequences $(p_n,q_n)_n$ and $(p'_n,q'_n)_n$, we could have chosen different sequences $(r_n)_n$ and $(s_n)_n$. By \cref{common multi}, it is enough to check what happens if we multiply $(s_n)_n$ and $(r_n)_n$ on the right by the same asymptotically injective sequence. But then it is clear that the output of our multiplication is in the same $\sim$-equivalence class. Second, we could have used different sequences to represent the equivalence class of $(p_n,q_n)_n$. Again, it is enough to check what happens if we multiply $(p_n)_n$ and $(q_n)_n$ by the same asymptotically injective sequence $(x_n)_n$.
Using \cref{approx Ore}, we get an asymptotically injective sequence $(s'_n)_n$ and a sequence $(r'_n)_n$ such that
\[
p'_n s'_n = q_n x_n r'_n.
\]
This is the same as choosing   $(x_n r'_n)_n$ and $(s'_n)_n$
instead of  $(r_n)_n$ and $(s_n)_n$, and we already know that such a change does not affect the final output. The last case, in which we multiply $(p'_n)_n$ and $(q'_n)_n$ by the same asymptotically injective sequence, is analogous.

Finally, $(1,1)_n$ is obviously a neutral element of the multiplication. We are left with distributivity. Consider three elements $(p_n,q_n)_n$, $(p'_n,q'_n)_n$, and $(p''_n,q''_n)_n$ of $\wrat(K,t)$. We need to verify that
\[
\left( (p_n,q_n)_n + (p'_n,q'_n)_n \right) \cdot (p''_n,q''_n)_n = (p_n,q_n)_n \cdot (p''_n,q''_n)_n + (p'_n,q'_n)_n \cdot (p''_n,q''_n)_n.
\]
We have already checked that we may move freely within the equivalence classes, and so we may assume that the first two sequences are already brought to common denominators, that is, $q_n = q'_n$ for every $n$. Now let $(r_n)_n$ be a sequence, and $(s_n)_n$ an asympotically injective sequence, such that $p''_n s_n = q_n r_n$. Then the left-hand side of our equation becomes
\[
\left( (p_n + p'_n)r_n, q''_n s_n \right)_n
\]
and the right-hand side becomes
\[
\left( p_n r_n, q''_n s_n \right)_n + \left( p'_n r_n, q''_n s_n \right)_n.
\]
These two expressions yield equal elements in $\wrat(K,t)$.\end{markproof}
As we have verified all of the claims, the proof is complete.\end{proof}

Henceforth, we will always endow $\wrat(K,t)$ with the above ring structure.

Recall that if we are given a stabilising sequence in $\calu G$, we have the limit map $\lambda_{\calu G}$ described in \cref{limit} returning a single element in $\calu G$.
Similarly, given an admissible $K$-stabilising sequence in $\calu K [t^{\pm 1} \rrbracket$, we may apply the limit map over $\calu K$ in every degree separately, and obtain a map $\lambda_{\calu K [t^{\pm 1} \rrbracket}$ returning an element of $\calu K [t^{\pm 1} \rrbracket$.

\begin{prop}
	\label{Lambdas}
	 The map $\Lambda_{\calu G} \colon \wrat(K,t) \to \calu G$ obtained by composing
	 \[(p_n,q_n)_n \mapsto (p_n {q_n}^\dagger)_n\]
	  with $\lambda_{\calu G}$ is an injective ring homomorphism.
	
	Similarly,
	  the map $\Lambda_{\calu K [t^{\pm 1} \rrbracket} \colon \wrat(K,t) \to \calu K [t^{\pm 1} \rrbracket$ obtained by composing
	  \[{(p_n,q_n)_n \mapsto (p_n \overline{q_n})_n}\]
	   with $\lambda_{\calu K [t^{\pm 1} \rrbracket}$ is an injective ring homomorphism.
\end{prop}
\begin{proof}
	Let us start with $\Lambda_{\calu G}$. Take $(p_n,q_n)_n \sim (p'_n,q'_n) \in \wrat(K,t)$. By \cref{agreement}, we have $(p_n {q_n}^\dagger)_n \apa (p'_n {q'_n}^\dagger)_n$, and so  $\lambda_{\calu G}$ sends both to the same element of $\calu G$ by \cref{limit}. Hence $\Lambda_{\calu G}$ is well defined.
	
	To check that $\Lambda_{\calu G}$ is additive and multiplicative, we use \cref{apa} and the uniqueness part of \cref{limit}.
	Thus, $\Lambda_{\calu G}$  is a ring homomorphism.
	
	Finally, suppose that $\Lambda_{\calu G}\left( (p_n,q_n)_n \right) = 0$. Unravelling the definitions, we get $(p_n {q_n}^\dagger)_n \apa (0)_n$. Since $(q_n)_n$ is asymptotically injective, \cref{apa,agreement} give
	\[
	(0)_n \apa (0 \cdot q_n)_n \apa (p_n {q_n}^\dagger q_n)_n \apa (p_n)_n
	\]
	and so $(p_n,q_n) = (0,1)$ in $\wrat(K,t)$.
	
	The situation for the map $\Lambda_{\calu K [t^{\pm 1} \rrbracket}$ is completely analogous: we use \cref{apaK,ex is an inverse} instead of \cref{apa,agreement}, respectively.
\end{proof}

Using the two maps above, we may view $\wrat(K,t)$ as a subring of both $\calu G$ and $\calu K[t^{\pm 1} \rrbracket$.

\begin{lemma}
	\label{obvious inverse}
	If $(p_n)_n$ and $(q_n)_n$ are asymptotically injective sequences over $\calu K [t^{\pm 1}]$ and $(p_n,q_n)_n \in \wrat(K,t)$, then $(q_n,p_n)_n \in \wrat(K,t)$ as well, and $(p_n,q_n)_n \cdot (q_n,p_n)_n = (1,1)_n$.
\end{lemma}
\begin{proof}
	To check that $(q_n,p_n)_n \in \wrat(K,t)$, we need to show that $(q_n{p_n}^\dagger)_n$ stabilises and that $(q_n\overline{p_n})_n$ $K$-stabilises. The first follows immediately from the fact that $(p_n{q_n}^\dagger)_n$ stabilises, \cref{adjoints and daggers}, and the observation that ${{q_n}^\dagger}^\dagger = q_n$.
	
	For the second, we first observe that $\init (\overline{p_n}) = (\init p_n) ^\dagger$, and hence $(\overline{p_n})_n$ is an asymptotically injective sequence. Therefore so is $(q_n\overline{p_n})_n$. We then use \cref{ex is an inverse}.
	
	The last statement follows directly from the definition of multiplication.
\end{proof}

\begin{prop}
	\label{wrat is division closed}
	The subring
	\[\nov \QQ G \phi \cap \wrat(K,t)\]
	 of $\calu K [t^{\pm1} \rrbracket$ contains $\QQ G$ and is division-closed inside $\nov \QQ G \phi$.
\end{prop}
\begin{proof}
	\cref{ring struct} tells us that $\wrat(K,t)$ contains $\calu K[t^{\pm1}]$, which in turn contains $\QQ K[t^{\pm1}] = \QQ G$.
	
	Now take a Laurent power series $x = \sum_{i=k}^\infty t^i x_i \in \nov \QQ G \phi \cap \wrat(K,t)$ with $x_i \in \QQ K$ for all $i$, and suppose that it is invertible in $\nov \QQ G \phi$. Truncating the inverse we find $y \in \QQ G$ such that $\init (xy) = 1$. Since $x$ and $y$ lie in $\wrat(K,t)$, and since the latter object is a ring, we conclude that $xy \in \wrat(K,t)$ as well. We thus have $(p_n,q_n)_n \in \wrat(K,t)$ with $\Lambda_{\calu K [t^{\pm 1} \rrbracket}\left( (p_n,q_n)_n \right) = xy$.
	By multiplying by suitable powers of $t$, we easily arrange for the associated power of $\init q_n$ to be $0$, for every $n$.
	
	We now need to worry about a pathological situation in which the powers associated to $\init p_n$ are not $0$. It is easy to see that the power cannot be positive infinitely often, since then the sequence of terms next to $t^0$ in $p_n \overline{ q_n}$ could not be equal to $1$ in the limit. So the powers are eventually non-positive. This remains true for every other pair of sequences $\sim$-equivalent to $(p_n,q_n)_n$.
	
	 Since $p_n$ is admissible, these powers are bounded from below. We are going to modify $p_n$, staying in the same $\sim$-equivalence class, so that the $\liminf$ of the powers becomes zero.
	
	 Suppose that the lowest power that appears infinitely often is $k$. Since the terms next to $t^k$ in $p_n \overline{q_n}$ tend to $0$, we easily produce an asymptotically injective sequence $(r_n)_n$ of projections in $\caln K$ such that $p_n r_n$ has initial term with associated power greater than $k$ for every $n$. We also replace $q_n$ with $q_nr_n$, and obtain $(p_n,q_n)_n \sim (p_nr_n,q_nr_n)_n$. We repeat this process $|k|$-times, and arrive at $(p_n,q_n)_n \sim (p'_n,q'_n)_n$ with the initial terms of $q'_n$ all having associated power $0$, and with the corresponding power for $\init p'_n$ being eventually $0$.
	
	We claim that $(p'_n)_n$ is asymptotically injective. Once this is shown, \cref{obvious inverse} will tell us that $(p'_n,q'_n)_n$ is invertible in $\wrat(K,t)$, and hence so is $xy$, finishing the proof.
	
	To prove the claim, observe that for large enough $n$ we have
	\[\init (p'_n \overline{q'_n}) = \init p'_n (\init q'_n)^\dagger,\]
	since $(q'_n)_n$ is asymptotically injective.	
	If $(p'_n)_n$ is not asymptotically injective, then neither is \[(\init p'_n (\init q'_n)^\dagger)_n \apa (1)_n.\]
	 This is a contradiction.
\end{proof}

\section{The main results}\label{sec proofs}

\begin{prop}\label{prop complexes}
Let $G$ be a countable discrete group and let
\[\varphi\colon G\to \Z\]
 be a character.  Let $C_\bullet$ be a complex of free $\QQ G$-modules which is finitely generated up to degree $n$ and such that $C_k=0$ for $k<0$.  If
 \[H_j(C_\bullet\otimes_{\QQ G}\nov{\QQ}{G}{\varphi})=0\]
  for $j\leq n$, then $H_j(C_\bullet\otimes_{\QQ G}\calu G)=0$ for $j\leq n$.
\end{prop}
\begin{proof}
	By \cref{chain contractions}, we see that
	\[H_j(C_\bullet\otimes_{\QQ G} \cald (\QQ G \subset \nov \QQ G \phi))=0\]
	 for $j\leq n$. By \cref{wrat is division closed}, the ring $\cald (\QQ G \subset \nov \QQ G \phi)$ is a subring of $\wrat(K,t)$, where $K = \ker \phi$ and $t \in G$ denotes an element such that $\phi(t)$ generates $\im \phi = \Z$. Using a chain contraction up to dimension $n$ we conclude that
		\[H_j(C_\bullet\otimes_{\QQ G} \wrat(K,t) )=0\]
	for $j\leq n$. Now, \cref{Lambdas} tells us that we may view $\wrat(K,t)$ as a subring of $\calu G$, and hence again using chain contractions we see that
		\[H_j(C_\bullet\otimes_{\QQ G} \calu G)=0\]
	for $j\leq n$, and the proof is finished.
\end{proof}

We are now ready to prove the remaining results from the introduction.

\medskip

\setcounter{thmx}{0}
\begin{thmx}\label{thm groups}
    Let $G$ be a group of type $\mathsf{FP}_n(\QQ)$.  If $b_n^{(2)}(G)\neq0$, then $\Sigma^n(G;\QQ)=\emptyset$.
\end{thmx}

\begin{proof}
As $\Sigma^k(G;\QQ)\subseteq \Sigma^\ell(G;\QQ)$ for all $k\geq \ell$, we may assume $n$ is the lowest dimension for which $b_n^{(2)}(G)\neq 0$.  In particular, $H_n(G;\calu G)\neq 0$.
 Since $G$ is of type $\mathsf{FP}_n(\QQ)$, there is a resolution $C_\bullet$ of $\QQ$ by free $\QQ G$-modules that is finitely generated up to dimension $n$.
 We may compute $H_p(G;\calu G)$ as $H_p(C_\bullet\otimes_{\QQ G} \calu G)$.
 It follows that $H_n(C_\bullet \otimes_{\QQ G} \calu G)\neq0$.  Thus, by \Cref{prop complexes}, for every character $\varphi\colon G\onto \Z$, we have $H_n(G;\nov{\QQ}{G}{\varphi}) = H_n(C_\bullet \otimes_{\QQ G} \nov{\QQ}{G}{\varphi})\neq0$.
In particular, by openness of the BNSR invariants (\Cref{thm props BNSR}\eqref{thm props BNSR open}) and Sikorav's Theorem, $\Sigma^n(G;\QQ)=\emptyset$.
\end{proof}

The proof of \Cref{thm spaces} is entirely analogous once one replaces $C_\bullet$ with $C_\bullet(\widetilde{X};\QQ)$, the chain complex of the universal cover of $X$, viewed as a chain complex of free $\QQ \pi_1 X$ modules.

\medskip

\setcounter{thmx}{2}
\begin{corx}\label{cor manifolds}
Let $M$ be a closed connected $2n$-manifold or (more generally) a finite $\mathsf{PD}_{2n}(\QQ)$-complex.
If $\chi(M)\neq 0$,
then $\Sigma^n(M)=\Sigma^n(M;\QQ)=\emptyset$.
In particular, if $M$ is additionally aspherical, then
$\Sigma^n(\pi_1M)=\Sigma^n(\pi_1M;\QQ)=\emptyset$.
\end{corx}

\begin{proof}
	Let us start with $M$ being a manifold.
After passing to a finite cover we may assume that $M$ is orientable and by \cite[Theorem~III]{KirbySiebenmann1969} we may assume that $M$ has the homotopy type of a finite CW complex.  By \cite[Remark~6.81]{Lueck2002} we have \[\sum_{p\geq 0}(-1)^pb^{(2)}_p(\widetilde M;\pi_1 M)=\chi(M)\neq 0.\]  In particular, there is some $p$ where $b^{(2)}_p(\widetilde M;\pi_1 M)\neq0$.  By Poincar\'e duality \cite[Theorem~1.35(3)]{Lueck2002} we have $b^{(2)}_{2n-k}(M)=b^{(2)}_k(M)$.  In particular, we may assume that $p\leq n$.  Now, \Cref{thm spaces} implies that $\Sigma^p(M;\QQ)=\emptyset$, of which $\Sigma^n(M)$ is a subset by \Cref{thm props BNSR}\eqref{thm props BNSR htpy subset hmlgy}. The argument for a $\mathsf{PD}_{2n}(\QQ)$-complex is identical.

Suppose in addition that $M$ is aspherical.  Then, $M$ is a model for $K(\pi_1 M,1)$ and so $b_p^{(2)}(\pi_1 M)=b_p^{(2)}(\widetilde M;\pi_1M)$.  Now, by \Cref{thm groups} we have $\Sigma^p(\pi_1 M;\QQ)=\emptyset$. The result follows from \Cref{thm props BNSR}\eqref{thm props BNSR htpy subset hmlgy}.
\end{proof}

\begin{corx}\label{cor Singer}
Let $G$ be a $\mathsf{PD}_n(\QQ)$-group and let $k=\lceil n/2\rceil-1$.  If $\Sigma^{k}(G;\QQ)$ is non-empty, then the Singer Conjecture holds for $G$.
\end{corx}

\begin{proof}
Arguing by Poincar\'e duality as in the proof of \Cref{cor manifolds} it suffices to show that
$b^{(2)}_p(G)=0$
 for $p\leq k$.
Now, by hypothesis there exists
$\varphi\in\Sigma^k(G,\QQ)$.
So,
by \Cref{thm groups},
$b^{(2)}_p(G)=0$ for $p\leq k$.
\end{proof}

\section{The Atiyah Conjecture and locally-indicable groups} \label{sec Atiyah}
In this section we prove versions of \Cref{thm groups} and \Cref{thm spaces} in positive characteristic.  This relies on the existence of certain Hughes-free\footnote{No relation to the first author.} skew fields.

Let $R$ be a skew field and let $G$ be a group.  When it exists, we denote by $\cald_{RG}$ the \emph{Hughes-free skew field} of $RG$. We omit the technical definition of a Hughes-free skew field as we do not require it.  However, we note that if it exists it is unique up to an $RG$-algebra isomorphism.

\begin{remark}
It is conjectured that $\cald_{RG}$ exists for any skew field $R$ for all locally indicable groups \cite[Conjecture~1]{Jaikin2021}; it is known to exist for residually \{locally indicable amenable\} groups \cite[Corollary~1.3]{Jaikin2021}.  In particular, $\cald_{RG}$ exists for RFRS groups.
\end{remark}

\begin{defn}
A group $G$ is \emph{agrarian} over a ring $R$ if there exists a skew-field $\cald$ and a monomorphism $\psi\colon RG \to \cald$ of rings.  If $G$ is agrarian over $R$ and $X$ is a space with $\pi_1 X=G$, then we define the \emph{agrarian $\cald$-homology of $X$} to be
\[H_p^\cald(X) =  H_p(C_\bullet(\widetilde{X};R)\otimes_{RG}\cald)\]
where $\cald$ is viewed as an $RG$-$\cald$-bimodule via the embedding $RG\to \cald$.  We also define the \emph{agrarian $\cald$-homology of $G$} to be
\[H_p^\cald(G) =  \Tor^{RG}_p(R,\cald).\]
Since modules over a skew field have a canonical dimension function taking values in $\NN\cup\{\infty\}$ we may define
\[b_p^{\cald}(X) =  \dim_\cald H_p^\cald(X)\quad \text{and} \quad b_p^{\cald}(G) =  \dim_\cald H_p^\cald(G). \]
If $\cald_{RG}$ exists, then (up to $RG$-isomorphism) we have a canonical choice of $\cald$ for each skew-field $R$.
\end{defn}

\begin{ConjAtiyah}
Let $G$ be a torsion-free countable group.  Then, the ring $\cald(\CC G\subset \calu G)$ is a skew field.
\end{ConjAtiyah}

\begin{remark}
If $G$ satisfies the Atiyah Conjecture, then $\cald_{\CC G}$ exists and is isomorphic to $\cald(\CC G\subset \calu G)$.  This applies for instance to torsion-free subgroups of right-angled Artin and Coxeter groups \cite{LinnellOkunSchick2012}, torsion-free virtually special groups \cite{Schreve2014}, locally indicable groups \cite{JaikinLopez2020}, and more \cite{Linnel1993,FarkasLinnell2006,Jaikin2019}.
\end{remark}

\begin{thm}\label{thm char p}
Let $R$ be a skew-field and let $G$ be a group such that $\cald_{RG}$ exists.  Let $\varphi\in S(G)$.
\begin{enumerate}
\item Suppose that $G$ is of type $\mathsf{FP}_n(R)$.  If $b_n^{\cald_{RG}}(G)\neq 0$, then $\Sigma^n(G;R)=\emptyset$.\label{thm agrarian groups}
\item Let $X$ be a connected CW complex with finite $n$-skeleton.  If $b_n^{\cald_{RG}}({X})\neq 0$, then $\Sigma^n(X;R)=\emptyset$. \label{thm agrarian spaces}
\end{enumerate}
In particular, if $G$ satisfies the Atiyah Conjecture, then statements \eqref{thm agrarian groups} and \eqref{thm agrarian spaces} hold with $R=\CC$ and with $\ell^2$-Betti numbers replacing $\cald_{\CC G}$-agrarian Betti numbers.
\end{thm}
\begin{proof}
We proceed as in \cite[Theorem~5.10]{HughesKielak2022}.  Let $K=\ker\varphi$.  Let $\mathbb K$ be the skew-field of twisted Laurent series with variable $t$ and coefficients in the skew field $\cald_{RK}$;  here $t$ is an element of $G$ with $\varphi(t)=1$ and the twisting extends the conjugation action of $t$ on $K$.  This is possible since $\cald_{RK}$ is Hughes-free (see \cite{Jaikin2021} for an explanation).

We have two embeddings, firstly $\nov{R}{G}{\varphi}$ embeds into $\mathbb K$, and secondly $\cald_{RG}$ embeds into $\mathbb K$.  To see the first embedding, $\nov{R}{G}{\varphi}$ may be viewed as a ring of twisted Laurent series in $t$ with coefficients in $RH$.  The second embedding exists because $\cald_{RG}$ is Hughes-free and so isomorphic as an $RG$-module to $\cald(RK[t^{\pm1}]\subset\mathbb K)$ where we identify $RG$ with $RK[t^{\pm1}]$.  In particular, we may view $\mathbb{K}$ as a $\cald_{RG}$-module.

\begin{claim}\label{claim agrarian}
Let $C_\bullet$ be a chain complex of  finitely generated free $RG$-modules such that $C_i=0$ for $i<0$.  If $H_j(C_\bullet\otimes_{RG}\nov{R}{G}{\varphi})=0$ for $j\leq n$, then $H_j(C_\bullet\otimes_{RG}\cald_{RG})=0$.
\end{claim}
\begin{claimproof}[Proof of \Cref{claim agrarian}]
Since $H_j(C_\bullet\otimes_{RG}\nov{R}{G}{\varphi})=0$ for $j\leq n$ and $\nov{R}{G}{\varphi}\subset\mathbb K$ it follows using chain contractions that $H_j(C_\bullet\otimes_{RG}\mathbb{K})=0$ for $j\leq n$.  Now, $\mathbb K$ and $\cald_{RG}$ are skew-fields, and so
\[H_j(C_\bullet\otimes_{RG}\cald_{RG}) \otimes_{\mathbb K} \mathbb K =H_j(C_\bullet\otimes_{RG}  \mathbb K)=0,\]
 forcing $H_j(C_\bullet\otimes_{RG}\cald_{RG})=0$ for $j\leq n$, as claimed.
\end{claimproof}
The theorem is now proved either by taking $C_\bullet$ in the claim to be a free resolution of $G$, finitely generated up to degree $n$, in the case of \eqref{thm agrarian groups}; or taking $C_\bullet$ to be $C_\bullet(\widetilde X;R)$, viewed as a chain complex of free $RG$-modules, in the case of \eqref{thm agrarian spaces}.  Finally, the result follows from the appropriate version of Sikorav's Theorem and openness of the BNSR invariants (\Cref{thm props BNSR}\eqref{thm props BNSR open}).
\end{proof}

Finally, we remark that characteristic $p$ versions of \Cref{cor manifolds} and \Cref{cor Singer} can be formulated and proved for groups $G$ where $\cald_{\FF_p G}$ exists by almost verbatim arguments -- with the exception of substituting \Cref{thm spaces,thm groups} with \Cref{thm char p}.

\section{Some examples}\label{sec ex}
In this section we detail a number of examples that both complement results already in the literature and might be of independent interest.

An \emph{elementarily free group} is a group with the same first order theory as a free group. Every finitely generated such group is of type $\mathsf F$ (see \cite{BridsonTweedaleWilton2007}).  A \emph{poly-elementarily-free group} of length $n$ is a group $G$ which admits a subnormal filtration $\{1\}=N_0\triangleleft N_1\triangleleft\dots\triangleleft N_n=G$ with $N_{i}/N_{i-1}$ isomorphic to a finitely generated elementarily free group.  Note that poly-$\{$finitely generated free or surface$\}$ groups are poly-elementarily-free.

\begin{lemma}\label{Bno poly elefree}
Let $G$ be a poly-elementarily-free group of length $n$.  If $\chi(G)\neq 0$, then $b_p^{(2)}(G)=0$ for $p\neq n$ and $b_n^{(2)}(G)=|\chi(G)|$.
\end{lemma}

\begin{proof}
	The case $n=0$ is easily dealt with, and so we may assume that $n>0$.
	
	Let $(N_i)_i$ be a subnormal chain with every $N_i / N_{i-1} = G_i$ non-trivial and finitely generated elementarily free, and with $N_0 = \{1\}$ and $N_n = G$.
Note that each $N_i$ is a group of type $\mathsf{F}$ since it is an extension of such groups.  Since $\chi(G)\neq 0$ and Euler characteristic is multiplicative over short exact sequences of groups of type $\mathsf{F}$, it follows that $\chi(N_i)\neq0$ and $\chi(G_i)\neq 0$ for every $i$.  By \cite{BridsonKochloukova2017} we have $b^{(2}_p(G_i)=0$ unless $p=1$, in which case the first $\ell^2$-Betti number of $G_i$ may be positive.  Since $\chi(G_i)\neq 0$ we have that $b^{(2)}_1(G_i)>0$.

An inductive application of \cite[Theorem~6.67]{Lueck2002} yields that for every $i$ and every $p<i$,
\[b^{(2)}_p(N_i)=0.\]
   By \cite[Theorem~B]{BridsonTweedaleWilton2007}, an elementarily free group is measure equivalent to a free group.  Now, inductively applying \cite[Theorem 1.10]{SauerThom2010} shows that for every $i$ and every $p>i$ we have $b^{(2)}_i(N_i)=0$. Thus, $b_i^{(2)}(N_i)=|{\chi(N_i)}|$, and hence taking $i=n$ we obtain $b_n^{(2)}(G)=|{\chi(G)}|$.
\end{proof}

The following result generalises \cite[Proposition~1.5]{KochloukovaVidussi2022}, dealing with free-by-free or surface-by-surface groups, and the first part of \cite[Theorem~6.1]{KrophollerWalsh2019}, dealing with $\{$free group of rank $2$$\}$-by-free groups.

\begin{thm}\label{polycule}
Let $G$ be a poly-elementarily-free group of length $n$.  If $\chi(G)\neq 0$, then $\Sigma^n(G;\QQ)=\Sigma^n(G)=\emptyset$.
\end{thm}
\begin{proof}
The result now follows from \Cref{Bno poly elefree} and \Cref{thm groups}.
\end{proof}

We remark that the conclusions of \cref{Bno poly elefree,polycule} remain valid if $G$ is a poly-$\calx$ group where $\calx$ is the class of groups of type $\mathsf{FP}$ that are measure equivalent to a free group.

\begin{example}[Pure mapping class group of a punctured sphere]
Let $m\geq 3$, and let $S_m$ denote the $m$-punctured $2$-sphere.  Recall that the \emph{pure mapping class group} $\Gamma_m\coloneqq\mathrm{PMCG}(S_m)$ of $S_m$ is the group of mapping classes of $S_m$ which fix the $m$-punctures pointwise.  It is well known (see e.g.\ \cite[Section~9.3]{FarbMargalit2011}) that $\Gamma_m$ is poly-free of length $m-2$ and each subnormal quotient in the poly-free filtration is non-abelian.  Hence, $\chi(\Gamma_m)\neq 0$.  We have verified the hypotheses of \Cref{polycule} and conclude that $\Sigma^{m-2}(\Gamma_m;\QQ)=\Sigma^{m-2}(\Gamma_m)=\emptyset$.
\end{example}

\begin{example}[Real and complex hyperbolic lattices]
Let $\Gamma$ be a lattice in $\SO(2n,1)$ or $\SU(n,1)$.  By \cite{Dodziuk1979} (see also \cite{Borel1985}), we have  $b^{(2)}_p(\Gamma)=0$ except when $p=n$, in which case $b^{(2)}_n(\Gamma)\neq0$.  By \Cref{thm groups}, we have $\Sigma^n(\Gamma)=\Sigma^n(\Gamma;\QQ)=\emptyset$.  This result was already known for `simplest type lattices' in $\SU(n,1)$, see \cite{IsenrichPy2022}.
\end{example}

\bibliographystyle{alpha}
\bibliography{refs.bib}

\end{document}